\documentclass[11pt]{amsart}
\usepackage{amssymb}
\usepackage{color}

\usepackage[hmargin={20mm,20mm},vmargin={15mm,20mm}]{geometry}
\usepackage{geometry}
\usepackage[dvips]{graphicx}
\usepackage{graphicx}
\usepackage{color}
\bibstyle{plain}
\newtheorem{theorem}{Theorem}[section]
\newtheorem{fact}[theorem]{Fact}
\newtheorem{example}[theorem]{Example}

\newtheorem{lemma}[theorem]{Lemma}
\newtheorem{proposition}[theorem]{Proposition}
\newtheorem{corollary}[theorem]{Corollary}

\newtheorem{definition}[theorem]{Definition}
\newtheorem{remark}[theorem]{Remark}

\newcommand{\tOm}{\tilde{\Omega}}
\newcommand{\F}{\mathcal F}

\newcommand{\K}{\mathbb K}
\newcommand{\KX}{\K(X)}
\newcommand{\Xinf}{\ell_\infty(X)}
\newcommand{\KXinf}{\ell_\infty(\K(X))}

\newcommand{\mS}{\mathcal F}
\newcommand{\N}{\mathbb N}

\newcommand{\R}{\mathbb R}

\newcommand{\on}{\operatorname}
\newcommand{\gifs}{$\mbox{GIFS}_\infty$}

\subjclass[2010]{}

\author{\L ukasz Ma\'slanka}
\address{Institute of Mathematics, \L\'od\'z University of Technology, W\'olcza\'nska 215, 93-005
\L\'od\'z, Poland}
\email {lukasz$\_$maslanka@interia.pl}

\author{Filip Strobin}
\address{Institute of Mathematics, \L\'od\'z University of Technology, W\'olcza\'nska 215, 93-005
\L\'od\'z, Poland}
\email {filip.strobin@p.lodz.pl}
\title[On generalized iterated function systems defined on $\ell_\infty$-sum]{On generalized iterated function systems defined on $\ell_\infty$-sum of a metric space}
\subjclass[2010]{Primary: 28A80; Secondary: 37C25, 37C70} 
\keywords{iterated function systems, generalized iterated function systems, fractals, generalized fixed points, code spaces, Cantor sets}

\begin{document}
\begin{abstract}
Miculescu and Mihail in 2008 introduced a concept of a \emph{generalized iterated function system} (GIFS in short), a particular extension of classical IFS. Instead of families of selfmaps of a metric space $X$, they considered families of mappings defined on finite Cartesian product $X^m$. It turned out that a great part of the classical Hutchinson--Barnsley theory has natural counterpart in this GIFSs' case.\\
Recently, Secelean extended these considerations to mappings defined on the space $\ell_\infty(X)$ of all bounded sequences of elements of $X$ and obtained versions of the Hutchinson--Barnsley theorem for appropriate families of such functions.\\
In the paper we study some further aspects of Secelean's setting. In particular, we introduce and investigate a bit more restrictive framework and we show that some problems of the theory have more natural solutions within such a case.
Finally, we present an example which shows that this extended theory of GIFSs gives us fractal sets that cannot be obtained by any IFSs or even by any GIFSs.
\end{abstract}
\maketitle

\section{Introduction}
Let $(X,d)$ be a metric space. By $\K(X)$ (or $\K(X,d)$) we denote the space of all nonempty and~compact subsets of $X$, endowed with the Hausdorff-Pompeiu metric
$$
H^d(K,D):=\max\left\{\sup\{\inf\{d(x,y):x\in K\}:y\in D\},\sup\{\inf\{d(x,y):x\in D\}:y\in K\}\right\}.
$$
It is well known that $\K(X)$ is complete [compact], provided $X$ is complete [compact].\\
Now if $\F=\{f_1,...,f_n\}$ is a finite family of continuous selfmaps of $X$, then by the same letter $\F$ we define the map $\F:\K(X)\to\K(X)$ in the following way
$$
\F(K):=f_1(K)\cup...\cup f_n(K).
$$ 
Now let $\F=\{f_1,...,f_n\}$ be a finite family of Banach contractions of $X$. The well-known Hutchinson--Barnsley  theorem from early 1980's (\cite{Hut},\cite{Ba}) states that there exists a unique $A_\F\in\K(X)$ such that
$$
A_\F=\F(A_\F)=f_1(A_\F)\cup...\cup f_n(A_\F)
$$
and, moreover, for every $K\in\K(X)$, the sequence $(\F^{(k)}(K))$ of iterations $\F$ of the set $K$, converges to $A_\F$ with respect to the Hausdorff-Pompeiu metric. The set $A_\F$ is called \emph{a fractal} or \emph{an attractor} generated by $\F$. In this setting, a finite family of continuous selfmaps of $X$ is called \emph{iterated function system} (IFS in short).\\
The H--B theorem can be considered as a milestone of a very important part of fractals theory, which is now really rich and advanced. One of the direction of further studies undertaken by mathematicians dealt with the question:
\begin{center}What should be assumed on a function system $\F$ so that it still generates a fractal set?
\end{center}
In particular, it turned out that instead of Banach contractions, it is enough to assume much weaker conditions (see \cite{Ha}) and that IFSs and their fractals can be defined in a purely topological way (see \cite{BKNNS}, \cite{Mi1}, \cite{Ka}).\\
An interesting version of H--B theory was introduced by Miculescu and Mihail in 2008 (see \cite{M}, \cite{MM1}, \cite{MM2}). The main idea was to consider, instead selfmaps of $X$, mappings defined on finite Cartesian product $X^m$ with values in $X$. It turned out that a great part of the H--B theory of fractals has a natural counterpart in such generalized case and also, the class of fractals in such general setting is essentially wider than the class of classical IFSs' fractals (see \cite{S}).\\
Then, in 2014, Secelean (see \cite{Se}) considered mappings defined on $\ell_\infty(X)$, i.e. on $\ell_\infty$-sum of space $X$, and proved a version of the H--B theorem for families of such mappings.
However, in some sense the Secelean's approach is too wide - for example, it seems that the iteration procedure he considers is not a very natural counterpart of the Miculescu and Mihail case, and also some additional technical assumptions must be done to handle the theory.\\
The aim of our paper is to investigate some further aspects of Secelean's theory, and also to present and study a bit more restrictive setting (in which some problems will have more natural solutions). The~content is organized as follows:\\
In the next section we recall the frameworks of generalized IFSs due to Miculescu and Mihail (called GIFSs), and due to Secelean. We also recall a certain fixed point theorem which will be used in the main part, and prove some auxiliary results.\\
Section 3 is devoted to presenting basic setting of GIFSs$_\infty$ (that is, families of mappings defined on~$\ell_\infty(X)$). In particular, we introduce a more restrictive conditions than Secelean's one and prove the counterpart of the H--B theorem in such a case. We will also compare our result with the one of Secelean.\\
In Section 4 we will define a code space for GIFSs$_\infty$ and investigate some aspects of it.\\
Then, in Section 5 we prove that the relationships between GIFSs$_\infty$ and their code spaces are analogous to the classical case.\\
Finally, in Section 6 we present an example of a fractal generated by some GIFS$_\infty$ which cannot be obtained by any GIFS in the sense of Miculescu and Mihail.\\
Also, at each section, we will try to explain advantages of "our" more restrictive assumptions.
\section{Basic definitions and overview of known results}\label{section2}
\subsection{Generalized iterated function systems}
If $m$ is a natural number, then let $X^m$ be the Cartesian product of $m$ copies of $X$, endowed with the maximum metric. A finite family $\F~=~\{f_1,...,f_n\}$ of continuous mappings $X^m\to X$ is called a \emph{generalized iterated function systems of order~$m$} (GIFS in short). A GIFS $\F=\{f_1,...,f_n\}$ induces the map $\F:\K(X)^m\to \K(X)$ given by
$$
\F(K_0,...,K_{m-1}):=f_1(K_0\times...\times K_{m-1})\cup...\cup f_n(K_0\times...\times K_{m-1}).
$$
Finally, we say that $f:X^m\to X$ is \emph{a generalized Banach contraction}, if the Lipschitz constant $\on{Lip}(f)<1$.
Miculescu and Mihail proved the following version of the H--B theorem:
 \begin{theorem}\label{gifsinitial}
Assume that $(X,d)$ is complete, $m\in\N$ and $\F=\{f_1,...,f_n\}$ is a GIFS of order $m$ consisting of generalized Banach contractions. Then there is a unique set $A_\F\in\K(X)$ such that
\begin{equation}\label{ffinal1}
A_\F=\F(A_\F,...,A_\F)=\bigcup_{i=1}^nf_i(A_\F\times...\times A_\F).
\end{equation}
Moreover, for every $K_0,...,K_{m-1}\in \K(X)$, the sequence $(K_k)$ defined by \begin{equation}\label{ffinal2}K_{k+m}:=\F(K_{k},...,K_{k+m-1})\end{equation}
converges to $A_\F$ with respect to the Hausdorff-Pompeiu metric on $\K(X)$.
\end{theorem}
Observe that conditions (\ref{ffinal1}) and (\ref{ffinal2}) are natural counterparts of those in the classical Hutchinson--Barnsley theorem. However, if $m>1$, then they are more involving - for example, the procedure of iteration (\ref{ffinal2}) looks back for $m$ steps, not just one.\\
The set $A_\F$ in the above result is called \emph{the fractal} or \emph{the atractor} generated by $\F$.\\
The proof of Theorem \ref{gifsinitial} is based on the following counterpart of the Banach fixed point principle:
\begin{theorem}\label{new1}
If $X$ is a complete metric space, $m\in\N$ and $g:X^m\to X$ is a generalized Banach contraction, then $g$ has a unique generalized fixed point, i.e., a unique $x_*\in X$ such that
$$
g(x_*,...,x_*)=x_*.
$$
Moreover, for  every $x_0,...,x_{m-1}\in X$, the sequence $(x_k)$ defined by \begin{equation}\label{new2}x_{k+m}:=g(x_{k},...,x_{k+m-1})\end{equation}
converges to $x_*$.
\end{theorem}
\noindent We use Theorem \ref{new1} for the mapping $\F:\K(X)^m\to \K(X)$, which turns out to be a generalized Banach contraction.

Theorem \ref{gifsinitial} was proved by Mihail and Miculescu in \cite{M} and \cite{MM2} (see also \cite{MM1} for the case of compact $X$). Then, in \cite{SS1}, Strobin and Swaczyna extended it (and, in connection, also Theorem~\ref{new1}) to \emph{weaker} types of generalized contractions (i.e., which satisfy weaker contractive conditions) analogous to those introduced by Browder \cite{Br} and even to those by Matkowski \cite{Mat}.  Also, Strobin proved in \cite{S} that the class of GIFSs' fractals is essentially wider than the class of IFSs' fractals by showing appropriate examples on the plane.\\
In fact, it turns out that a great part of the classical IFS theory has a natural counterpart in this GIFS's setting - see references in mentioned papers and other articles of Miculescu, Mihail, Strobin, Secelean and their coauthors.\\
\subsection{Generalized iterated function systems on the $\ell_\infty$-sum - Secelean's approach}

Secelean in \cite{Se} (2014) considered GIFSs on $\ell_\infty$-sum of a space $X$. We will present here a particular version of his results. The difference is that we restrict here to the case of generalized Banach contractions, whereas \cite{Se} deals with more general contractive conditions. Nevertheless, the ideas are the same. Also, we consider here $I=\N^*=\{0,1,2,...\}$ and Secelean considered any $I\subset \N$. However (as was also remarked by Secelean), the case of infinite $I$ is essentially the same as $I=\N^*$.\\
 Given a metric space $(X,d)$, let $\ell_\infty(X)$ be the $\ell_\infty$-sum of $X$, i.e.,
\begin{equation*}
\Xinf:=\{(x_k)\subset X:(x_k)\;\mbox{is bounded}\},
\end{equation*}
and endow it with the supremum metric $d_s$:
\begin{equation}\label{se0}
d_{s}((x_k),(y_k)):=\sup\{d(x_k,y_k):k\in\N^*\}.
\end{equation}
Note that throughout the paper we mostly enumerate sequences by nonnegative integers. Thus when writing $(x_k)$, we automatically assume that $(x_k)=(x_k)_{k\in\N^*}$.\\
Let us notice that the notion of the $\ell_\infty$-sum of a family of spaces originates from functional analysis; see, e.g., \cite{LT}.
\begin{remark}\emph{Clearly, if $X$ is bounded, then $\ell_\infty(X)$ is just the Cartesian product: 
\begin{equation*}
\ell_\infty(X)=\prod_{k=0}^\infty X:=X\times X\times...,
\end{equation*}
but if $X$ is unbounded, then $\Xinf$ is a proper subspace of $\prod_{k=0}^\infty X$.}
\end{remark}
If $f:\ell_\infty(X)\to X$, then we define $\tilde{f}:X\to X$ by
\begin{equation}\label{se0,5}
\tilde{f}(x):=f(x,x,...).
\end{equation}
A point $x_*\in X$ will be called \emph{a generalized fixed point of $f$}, if $x_*$ is a fixed point of $\tilde{f}$, that is, if $f(x_*,x_*,...)=x_*$.\\
Secelean started with the fixed point theorem (see \cite[Theorem 3.1]{Se}):
\begin{theorem}\label{se1}
Assume that $X$ is a complete metric space and $f:\ell_\infty(X)\to X$ is such that the Lipschitz constant $\on{Lip}(f)<1$. Then $f$ has a unique generalized fixed point $x_*\in X$.\\
Moreover, for every $x=(x_k)\in \ell_\infty(X)$, the sequence $(y_k)$ defined by
\begin{equation}\label{aa2}
y_k:=f\left(\tilde{f}^{(k)}(x_0),\tilde{f}^{(k)}(x_1),\tilde{f}^{(k)}(x_2),...\right)
\end{equation} 
converges to $x_*$. More precisely, for any $k\in\N^*$,
$$
d(x_*,y_k)\leq \frac{\on{Lip}(f)^{k+1}}{1-\on{Lip}(f)}\sup\{d(\tilde{f}(x_i),x_i):i\in\N^*\}.
$$
\end{theorem}
On one hand, this result can be viewed as a generalization of the Banach fixed point theorem or Theorem \ref{new1}. On the other hand, it seems that the iteration procedure (\ref{aa2}) is not a very natural counterpart of (\ref{new2}).

Following \cite{Se}, we say that a mapping $f:\ell_\infty(X)\to X$ satisfies the condition (C1), if
\begin{center}
$(C1)\;\;\;$ for every $(K_k)\in\ell_\infty(\K(X))$, the closure of the image of the product $\overline{f(\prod_{k=0}^\infty K_k)}\in\K(X)$,
\end{center}
and condition (C2), if
\begin{center}
$(C2)\;\;\;$ for every $(K_k)\in\ell_\infty(\K(X))$, the image of the product ${f(\prod_{k=0}^\infty K_k)}\in\K(X).\;\;\;\;\;\;\;\;\;\;\;\;\;\;\;\;\;\;\;\;\;\;\;$
\end{center}
\begin{remark}\emph{If $(K_k)\in\KXinf$, then $\bigcup_{k=0}^\infty K_k$ is bounded, so $\prod_{k=0}^\infty K_k\subset\Xinf$. Hence the image $f(\prod_{k=0}^\infty K_k)$ is well defined.}
\end{remark}
 Finally, if $\F=\{f_1,...,f_n\}$ is a family of maps $\ell_\infty(X)\to X$ which satisfy (at least) (C1) condition, then we can define the map $\F:\ell_\infty(\K(X))\to\K(X)$ by
$$
\F(K_0,K_1,..):=\overline{f_1\left(\prod_{k=0}^\infty K_k\right)}\cup...\cup \overline{f_n\left(\prod_{k=0}^\infty K_k\right)}.
$$ 
Theorem \ref{se1} was used to obtain the following fractal theorem (see \cite[Theorem 3.7]{Se}):
\begin{theorem}\label{se2}
Assume that $X$ is a complete metric space and $\F=\{f_1,...,f_n\}$ is a family of maps which satisfy (C1) condition and such that $L(\F):=\max\{\on{Lip}(f_i):i=1,...,n\}<1$.
Then there is a~unique $A_\F\in\K(X)$ such that
$$
A_\F=\F(A_\F,A_\F,...).
$$
Moreover, for every $(K_k)\in\ell_\infty(\K(X))$, the sequence $(Y_k)$ defined by 
\begin{equation}\label{abc21}
Y_k:=\F\left(\tilde{\F}^{(k)}(K_0),\tilde{\F}^{(k)}(K_1),...\right)
\end{equation}
(where $\tilde{\F}(K):=\F(K,K,K,...)$), converges to $A_\F$ with respect to the Hausdorff-Pompeiu metric. More precisely, for every $k\in\N^*$,
$$
H^d(A_\F, Y_k)\leq \frac{L(\F)^{k+1}}{1-L(\F)}\sup\{H(\tilde{\F}(K_i),K_i):i\in\N^*\}.
$$
\end{theorem}
The set $A_\F$ is called \emph{a fractal generated by $\F$}. Note that
the result follows from Theorem \ref{se1} since it turns out that $\on{Lip}(\F)<1$.
\begin{remark}\emph{Let us remark that condition (C1) is important since, in general, the set $\overline{f(\prod_{k=0}^\infty K_k)}$ may not be compact, even if $\on{Lip}(f)<1$. The reason is that the product $\prod_{k=0}^\infty K_k$ may not be compact in the space $(\Xinf,d_s)$ (later we will discuss this issue in details).
}\end{remark}
\subsection{Another fixed point theorem for maps defined on $\ell_\infty(X)$}
Together with Jachymski, in \cite{JMS} we proved another version of Theorem \ref{new1} for maps defined on $\Xinf$. Let us present it.\\
At first, consider alternative metrics on the space $\ell_\infty(X)$. Namely, if $q\in(0,1]$, then set
\begin{equation}\label{s-metric}d_{s,q}(x,y) := \operatorname{sup}\{q^kd(x_k, y_k):k\in\N^*\} \;\; \textrm{ for any } x=(x_k), y=(y_k)\in \Xinf.\end{equation}
If additionally $q<1$ and $p\in[1,\infty)$, the we also define
\begin{equation}\label{p-metric}d_{p,q}(x,y) := \left( \sum_{k=0}^{\infty} q^kd^p(x_k, y_k) \right)^{1/p} \textrm{ for any } x=(x_k), y=(y_k)\in \Xinf . \end{equation}
Clearly, the metric $d_s$ considered by Secelean is exactly $d_{s,1}$. Note that if in (\ref{p-metric}) we assume that $q=1$, then we can get the value $\infty$. Hence, if we write $d_{p,q}$, then we will automatically assume that $q<1$ and $p\in[1,\infty)$ and if we write $d_{s,q}$, we will assume that $q\in(0,1]$, unless stated otherwise.

The following result lists basic properties of the space $\ell_\infty(X)$ endowed with $d_{p,q}$ or $d_{s,q}$ (see \cite[Propositions 2.2, 2.4 and 2.8]{JMS}).

\begin{proposition}\label{fp1}
In the above frame, assume that $q<1$ and $p\geq 1$. Then
\begin{itemize}
\item[(i)] $d_{s,q}\leq d_{p,q^p}$;
\item[(ii)] if $q\leq q'\leq 1$, then $d_{s,q}\leq d_{s,q'}$;
\item[(iii)] if $q^{{1}/{p}}<q'\leq 1$, then $d_{p,q} \leq \left(1-\frac{q}{(q')^p}\right)^{-1/p} d_{s, q'}$;
\item[(iv)] if $X$ is bounded, then the topology on $\Xinf$ ($=\prod_{k=0}^\infty X$), induced by any of metrics $d_{s,q}$ and $d_{p,q}$ is exactly the Tychonoff product topology.
\end{itemize}\end{proposition}
\begin{remark}\label{remark1}\emph{
(1) Part (iv) reveals probably the crucial difference between  metrics $d_{s,q}$ and $d_{p,q}$ for $q<1$, and the metric $d_{s,1}$ considered by Secelean. Later we will see that this difference has many consequences.\\ 
(2) At the end of \cite{Se}, Secelean considered also the metric $d_{1,\frac{1}{2}}$ and observed that such metric has nice properties. In fact, some of our observations are related to that remark.}
\end{remark}
Now we present the fixed point theorem from \cite{JMS}. We need, however, some further background.
Let $f:\Xinf\to X$, and $x=(x_k)\in \Xinf$. Define the sequence $(x^k)\subset X$ by the following inductive formula:
\begin{equation}\label{fili4}
x^1:=f(x_0,x_1,...),
\end{equation}
\begin{equation}\label{fili5}
x^{k+1}:=f(x^k,...,x^1,x_0,x_1,...),\;\;k\geq 1.
\end{equation}
Observe that this iteration procedure has many similarities with (\ref{new2}). Hence we say that the sequence $(x^k)$ is \emph{the sequence of generalized iterates of $f$} of the sequence $x$. \\
Now if $f:\Xinf\to X$, then by $L_{s,q}(f)$, $L_{p,q}(f)$ let us denote the Lipschitz constants of $f$ with respect to metrics $d_{s,q}$ and $d_{p,q}$, respectively (of course, we allow them to be equal to $\infty$).\\
The following result (\cite[Theorem 3.7]{JMS}) is a counterpart of Theorem \ref{new1}; in fact, as was proved in the last section of \cite{JMS}, it implies Theorem \ref{new1}.
\begin{theorem}\label{Banach_inf}
Let $(X, d)$ be a complete metric space and let $f: \Xinf\to X$ be such that one of the conditions holds
\begin{equation}\label{fp1,5}
L_{s,q}(f)<1\;\;\;\mbox{for some}\;q\in(0,1);
\end{equation}
\begin{equation}\label{fp1,1}
L_{p,q}(f)<(1-q)^{1/p}\;\;\;\mbox{for some}\;q\in(0,1),\;p\in[1,\infty).
\end{equation}
Then $f$ has a unique generalized fixed point $x^*\in X$.\\ Moreover, for any $(x_k)\in \Xinf$, the sequence of generalized iterations $(x^k)$ converges to $x^*$. More precisely:
\begin{itemize}
\item[-] if $L_{s,q}(f)<1$, then
$$d(x^k, x^*)\leq L_{s,q}(f) \frac{(\max\{L_{s,q}(f), q \})^{k-1}}{1-\max\{L_{s,q}(f), q \}} d_{s,q}\left((x_0,x_1,x_2...),(x^1,x_0,x_1,...)\right);$$
\item[-] if $L_{p,q}(f)<(1-q)^{1/p}$, then
$$d(x^k, x^*)\leq L_{p,q}(f)\frac{((L_{p,q}(f))^p+q)^{\frac{k-1}{p}}}{1-((L_{p,q}(f))^p+q)^{\frac{1}{p}}} d_{p,q}\left((x_0,x_1,x_2...),(x^1,x_0,x_1,...)\right).$$

\end{itemize}
\end{theorem}

\begin{remark}\label{abc7}\emph{
The condition (\ref{fp1,5}) looks more natural than (\ref{fp1,1}). However, as was observed in \cite[Remark 3.8]{JMS}, they are equivalent. We will need an extended version of this observations, so we give here short explanation:\\
By Proposition \ref{fp1}(i) it follows that 
\begin{equation}\label{abc22}L_{p,q^p}(f)\leq L_{s,q}(f).\end{equation} Moreover it is easy to see that for $q<1$,
\begin{equation}\label{fp2}
\lim_{p\to\infty}(1-q^p)^{1/p}=1.
\end{equation}
Hence if $L_{s,q}(f)<1$ (where $q<1$), then for some $p\in[1,\infty)$, $L_{p,q^p}(f)<(1-q^p)^{1/p}$. \\
Conversely, by Proposition \ref{fp1}(iii) we get
\begin{equation}\label{lukasz1}L_{s,q'}(f) \leq \left(1-\frac{q}{(q')^p}\right)^{-1/p} L_{p,q}(f) \textrm{ for $q' > q^{1/p}$}.\end{equation}
Hence, if $L_{p,q}(f) < (1-q)^{1/p}$, then for $q'$ close enough to $1$, we have $L_{s, q'} (f)< 1$.
}
\end{remark}

\begin{remark}\label{se4}\emph{
By Proposition \ref{fp1}(ii), we have that $L_{s,1}(f)\leq L_{s,q}(f)$ for $q<1$. 
Hence if the function $f$ satisfies (\ref{fp1,5}) (or, equivalently, (\ref{fp1,1})) then it also satisfies assumptions of Theorem \ref{se1} (i.e., $L_{s,1}(f)<1$), but the converse need not be true. Indeed, the function $f:\ell_\infty([0,1])\to[0,1]$ defined by $$f((x_k)):=\frac{1}{2}\sup\{x_k:k\in\N^*\}$$ satisfies $L_{s,1}(f)=\frac{1}{2}<1$, but the sequence of generalized iterates $(x^k)$ does not converge for any sequence $(x_k)$ so that $x_i>0$ for some $i\in\N^*$  (see \cite[Example 3.11]{JMS} for details). 
}\end{remark}
It turns out that, under assumptions of Theorem \ref{Banach_inf}, the fixed point $x^*$ of the map $f$ is a limit of fixed points of certain restrictions of $f$ (see \cite[Theorem 3.13]{JMS}):
\begin{theorem}\label{fp5}
Assume that $f:\Xinf\to X$ satisfies (\ref{fp1,5}) (or, equivalently, (\ref{fp1,1})). Let ${x}\in X$ and $m\in\N$, and define ${f}_m:X^m\to X$ by
$$
{f}_m(x_0,...,x_{m-1}):=f(x_0,...,x_{m-1},{x},{x},...).
$$
Then the functions $f_m$ satisfy the assumptions of Theorem \ref{new1}, and the sequence $(x^*_m)$ of generalized fixed points of~$f_m$s' converges to $x^*$, a generalized fixed point of $f$.
\end{theorem}
\begin{remark}\label{abc8}\emph{The map $f$ from Remark \ref{se4} shows that the thesis of Theorem \ref{fp5} does not hold if we just assume that $L_{s,1}(f)<1$, i.e., under the assumptions of Theorem \ref{se1} - see \cite[Example 3.14]{JMS}.}
\end{remark} 


\section{Basic properties and auxiliary constructions}
In this section we make some initial observations and introduce general constructions which will be used later.

\subsection{The space $\ell_\infty(\K(X))$}
Let $(X,d)$ be a metric space. Then on one hand we can consider the space $\K(\Xinf)$ with the Hausdorff-Pompeiu metric induced by considered metric on $\Xinf$ (i.e., $H^{d_{p,q}}$ or $H^{d_{s,q}}$), and on the other, we can consider the space $\KXinf$ with the metrics induced by the Hausdorff-Pompeiu metric on $\K(X)$ (i.e., $H^d_{p,q}$ or $H^d_{s,q}$).\\
As was observed by Secelean, 
the space $(\K(\Xinf),H^{d_{s,1}})$ is far from $(\KXinf,H^d_{s,1})$ since for a set $K\subset X$, the product $\prod_{k=0}^\infty K:=K\times K\times...$ belongs to $\K(\Xinf,d_{s,1})$ iff $K$ is a singleton, but on the other hand, $(K,K,...)\in\KXinf$ if $K\in\K(X)$. In connection with it, the condition (C1) was needed in Theorem \ref{se2}.\\
As we will see in the next result, such problems do not appear in case of metrics $d_{p,q}$ and $d_{s,q}$ when $q<1$. Also, for completeness, we extend the observation of Secelean. 
\begin{theorem}\label{Hinf=H}
Let $(X, d)$ be a metric space.\\
(1) If $(K_k)$ is a sequence of subsets of $X$, $p\in[1,\infty)$ and $q\in(0,1)$, then the following conditions are equivalent:
\begin{itemize}
\item[(i)] $(K_k)\in \KXinf$;
\item[(ii)] $\prod_{k=0}^\infty K_k\in \K(\Xinf,d_{s,q})$;
\item[(iii)] $\prod_{k=0}^\infty K_k\in \K(\Xinf,d_{p,q})$.
\end{itemize}
(2) If $(K_k)$ is a sequence of subsets of $X$, then the following conditions are equivalent:
\begin{itemize}
\item[(i)] each $K_k$ is compact and the sequence of diameters $\on{diam}_d(K_k)\to 0$;
\item[(ii)] $\prod_{k=0}^\infty K_k\in \K(\Xinf,d_{s,1})$.
\end{itemize}
(3) For every sequences $(K_k),(D_k)\in \KXinf$:\\
- if $q\in(0,1)$ and $p\in[1,\infty)$, then we have:
$$H^{d_{p,q}}\left(\prod_{k=0}^{\infty}K_k, \prod_{k=0}^{\infty}D_k\right) \leq H^d_{p,q}((K_k), (D_k)),$$
- if $q\in(0,1]$, then we have
$$H^{d_{s,q}}\left(\prod_{k=0}^{\infty}K_k, \prod_{k=0}^{\infty}D_k\right)=H^d_{s,q}((K_k), (D_k))  $$
\end{theorem}
(in case $q=1$, $\prod_{k=0}^\infty K_k$ may not be compact, but we clearly can calculate the value $H^{d_{p,q}}(\prod_{k=0}^{\infty}K_k, \prod_{k=0}^{\infty}D_k)$).
\begin{proof}Ad(1) We will prove just the equivalence $(i)\Leftrightarrow (iii)$ (the equivalence $(i)\Leftrightarrow (ii)$ goes in the same way; in fact, it follows from the equivalence $(i)\Leftrightarrow (iii)$ and Proposition \ref{fp1}).\\
Let $\prod_{k=0}^\infty K_k \in \K(\Xinf,d_{p,q})$. 
Since the convergence in $d_{p,q}$ implies the convergence at each coordinate (see \cite[Proposition 2.3]{JMS}), we have that each projection $\on{proj}_k:\Xinf\to X$ is continuous. In particular, this means that sets $K_k$ are compact. 
Now we show that $\bigcup_{k=0}^\infty K_k$ is bounded. Assume on the contrary that $\bigcup_{k=0}^\infty K_k$ is unbounded. Since each $K_k$ is bounded (as a compact set), we can choose an unbounded sequence $(x_k)$ such that $x_k\in K_k$ for each $k=0,1,...$. In particular, $(x_k)\in\prod_{k=0}^\infty K_k\subset \Xinf$. This is a contradiction, which means that $\bigcup_{k=0}^\infty K_k$ is bounded and, in turn, $(K_k)$ is a bounded sequence of compact sets. Thus $(K_k)\in \KXinf$.\\
Now let $(K_k) \in \KXinf$ and set $K:=\bigcup_{k=0}^\infty K_k$. Then $K$ is bounded (as the sequence $(K_k)$ is bounded), hence by Proposition~\ref{fp1}(iv), the topology on $(\ell_\infty(K),d_{p,q})$ is the Tychonoff product topology. Now since each $K_k\subset K$, we have that $\prod_{k=0}^\infty K_k$ is compact in $(\ell_\infty(K),d_{p,q})$ which also means that $\prod_{k=0}^\infty K_k\in \K(\Xinf,d_{p,q})$.\\
Ad(2) Let $(K_k)$ be a sequence of subsets of $X$. Assume first that each $K_k$ is compact and $\on{diam}_d(K_k)\to 0$. By definition of the metric $d_{s,1}$ it can be easily seen that the topology on $\prod_{k=0}^\infty K_k$ is exactly the Tychonoff topology. Hence $\prod_{k=0}^\infty K_k$ is compact subset of $(\Xinf,d_{s,1})$.\\
Now assume that $\prod_{k=0}^\infty K_k\in\K(\Xinf,d_{s,1})$. Again we can observe that each projection $\on{proj}_k:~\Xinf\to~X$ is continuous. Hence each $K_k$ is compact.
Now suppose that $\on{diam}_d(K_k)$ does not converge to $0$, and choose $\varepsilon>0$ and $n_0<n_1<...$ so that $\on{diam}_d(K_{n_k})> \varepsilon$. For every $n=0,1,...$, choose $x^n,y^n\in~K_n$ such that $d(x^n,y^n)\geq \frac{1}{2}\on{diam}_d(K_n)$ and, finally, for any $k=0,1,...$, set $z_k=(z_k^n)$ so that $z_k^n:=\left\{\begin{array}{ccc}y^{n_k}&\mbox{if}&n=n_k\\x^n&\mbox{if}&n\neq n_k\end{array}\right.$. Then $(z_k)$ does not have a convergent subsequence. It is a contradiction.\\
Ad(3)
Let $(K_k),(D_k)\in \KXinf$. At first, consider the metric $d_{p,q}$. Take any $(x_k)\in\prod_{k=0}^{\infty}K_k$. Then we have
$$
\inf_{(y_k)\in \prod_{k=0}^\infty D_k}\left( \sum_{k=0}^{\infty} q^k d^p(x_k, y_k) \right)^{1/p}=\left( \sum_{k=0}^{\infty} q^k \inf_{y_k\in D_k}d^p(x_k, y_k) \right)^{1/p}=  \left( \sum_{k=0}^{\infty} q^k\left(\inf_{y_k\in D_k}d(x_k, y_k)\right)^p\right)^{1/p}
$$
$$
\leq\left( \sum_{k=0}^{\infty} q^k(H^d)^p(K_k,D_k)\right)^{1/p}=H^d_{p,q}((K_k),(D_k)).
$$
Hence
$$
\sup_{(x_k)\in\prod_{k=0}^\infty K_k}\left(
\inf_{(y_k)\in \prod_{k=0}^\infty D_k}\left( \sum_{k=0}^{\infty} q^k d^p(x_k, y_k) \right)^{1/p}\right)\leq H^d_{p,q}((K_k),(D_k)).
$$
In the same way we show
$$
\sup_{(y_k)\in \prod_{k=0}^\infty D_k}\left(\inf_{(x_k)\in \prod_{k=0}^\infty K_k}\left( \sum_{k=0}^{\infty} q^k d^p(x_k, y_k) \right)^{1/p}\right)\leq H^d_{p,q}((K_k),(D_k)),
$$
which gives us 
$$H^{d_{p,q}}\left(\prod_{k=0}^\infty K_k,\prod_{k=0}^\infty D_k\right)\leq H^d_{p,q}((K_k),(D_k)).$$
Now consider the metric $d_{s,q}$. In a similar way as in the previous case, we can show the inequality
$$
H^{d_{s,q}}\left(\prod_{k=0}^\infty K_k,\prod_{k=0}^\infty D_k\right)\leq H^d_{s,q}((K_k),(D_k)).
$$
We will show the opposite inequality. Let $\varepsilon > 0$. There exists $k_0\in\N^*$ such that
$$
H^d_{s,q}((K_k),(D_k)) - \varepsilon \leq q^{k_0}H^d(K_{k_0},D_{k_0}).
$$
Now let $(x_k)\in \prod_{k=0}^\infty K_k$. Then, clearly,
$$
\inf_{(y_k)\in \prod_{k=0}^\infty D_k}d_{s,q}((x_k),(y_k))\geq q^{k_0}\inf_{y\in D_{k_0}} d(x_{k_0},y)
$$
so
$$
\sup_{(x_k)\in \prod_{k=0}^\infty K_k}\left(\inf_{(y_k)\in \prod_{k=0}^\infty D_k}d_{s,q}((x_k),(y_k))\right)\geq q^{k_0}\sup_{x\in K_{k_0}}\left(\inf_{y\in D_{k_0}}d(x,y)\right), $$
which gives us 
$$
H^{d_{s,q}}\left(\prod_{k=0}^\infty K_k,\prod_{k=0}^\infty D_k\right) \geq q^{k_0}\sup_{x\in K_{k_0}}\left(\inf_{y\in D_{k_0}}d(x,y)\right).$$
In the same manner, we show that
$$
\sup_{(y_k)\in \prod_{k=0}^\infty D_k}\left(\inf_{(x_k)\in \prod_{k=0}^\infty K_k}d_{s,q}((x_k),(y_k))\right)\geq q^{k_0}\sup_{y\in D_{k_0}}\left(\inf_{x\in K_{k_0}}d(x,y)\right),
$$
from which we obtain that: 
$$H^{d_{s,q}}\left(\prod_{k=0}^\infty K_k,\prod_{k=0}^\infty D_k\right) \geq q^{k_0}\sup_{y\in D_{k_0}}\left(\inf_{x\in K_{k_0}}d(x,y)\right).$$
Hence we arrive to
$$H^{d_{s,q}}\left(\prod_{k=0}^\infty K_k,\prod_{k=0}^\infty D_k\right) \geq q^{k_0} H(K_{k_0}, D_{k_0}) \geq H^d_{s,q}((K_k),(D_k)) - \varepsilon.$$
Since $\varepsilon$ was taken arbitrarily we finally get 
$$H^{d_{s,q}}\left(\prod_{k=0}^\infty K_k,\prod_{k=0}^\infty D_k\right) \geq H^d_{s,q}((K_k),(D_k)).$$
\end{proof}

\begin{example}\emph{
Define
$$
K_k:=\left\{\begin{array}{cc}[0,1]&\textrm{ if }k=0,\\
\{0\}&\textrm{ if }k\geq 1, \end{array}\right.\;\;\;\;\;\;
D_k:=\left\{\begin{array}{cc}[0,1]&\textrm{ if }k=1,\\
\{0\}&\textrm{ if }k\neq 1. \end{array}\right. 
$$
Clearly, $\prod_{k=0}^\infty K_k,\prod_{k=0}^\infty D_k\in \K(\ell_\infty(\R))$, and it is easy to see that
$
H^d_{p,q}((K_k),(D_k))=(1+q)^{1/p}
$ and $H^{d_{p,q}}(\prod_{k=0}^\infty K_k,\prod_{k=0}^\infty D_k)=\max\{1, q^{1/p}\} = 1$.
Hence in the case of the metric $d_{p,q}$, the inequality in the previous result cannot be replaced by the equality.}
\end{example}

\subsection{The sets $(X_k)$, spaces $(X^\infty_k)$, and metrics $d_{k,s,q}$ and $d_{k,p,q}$}$\;$\\
If $X$ is a nonempty set, then let $(X_k)$ be a sequence of sets defined by the following inductive formula:
\begin{equation}\label{abc0}
\begin{array}{ll}X_0:=X,\\
X_{k+1}:=\prod_{i=0}^\infty X_k:=X_k\times X_k\times...,\;\;\;\;\;k\geq 0.
\end{array}
\end{equation}
Now let $(X,d)$ be a metric space. Set
\begin{equation}\label{abc11}
X_0^\infty:=X\;\;\;\mbox{and}\;\;\;d_{0,s,1}:=d.
\end{equation}
Assume that we already defined the spaces $(X_k^\infty,d_{k,s,1})$ for some $k\in\N^*$. Define
\begin{equation}\label{abc12}
X^\infty_{k+1}:=\ell_\infty(X_k^\infty)\;\;\;\mbox{and}\;\;\;d_{k+1,s,1}:=(d_{k,s,1})_{s,1}.
\end{equation}
Now for $q\in(0,1)$ and $p\in[1,\infty)$, we will define metrics $d_{k,p,q}$ and $d_{k,s,q}$ on $X_k^\infty$.\\
On $X^\infty_0=X$, we set
\begin{equation}\label{abc13}
d_{0,p,q}:=d_{0,s,q}:=d.
\end{equation}
Now since $X^\infty_1=\ell_\infty(X)$, we can define
\begin{equation}\label{abc13,5}
d_{1,p,q}:=(d_{0,p,q})_{p,q}\;\;\mbox{ and }\;\;d_{1,s,q}:=(d_{0,s,q})_{s,q}.
\end{equation}
By Proposition \ref{fp1}(ii),(iii), we see that $d_{1,p,q}\leq \frac{1}{(1-q)^{1/p}}d_{1,s,1}$ and $d_{1,s,q}\leq d_{1,s,1}$.\\
Now assume that for some $k\in\N$, we defined metrics $d_{k,p,q}$ and $d_{k,s,q}$ on $X^\infty_k$ according to (\ref{s-metric}) and (\ref{p-metric}), i.e., they are defined by
\begin{equation}\label{abc14,5}
d_{k,p,q}:=(d_{k-1,p,q})_{p,q}\;\;\mbox{ and }\;\;d_{k,s,q}:=(d_{k-1,s,q})_{s,q}.
\end{equation}
and which satisfy $d_{k,p,q}\leq \frac{1}{(1-q)^{k/p}}d_{k,s,1}$ and $d_{k,s,q}\leq d_{k,s,1}$.\\ 
Then if a sequence $(x_i)$ of elements of $X^\infty_k$ is bounded with respect to $d_{k,s,1}$, then it is bounded with respect to metrics $d_{k,p,q}$ and $d_{k,s,q}$. Hence we can define $d_{k+1,p,q}$ and $d_{k+1,s,q}$ on $X^\infty_{k+1}$, according to (\ref{s-metric}) and (\ref{p-metric}), or, in other words, by
\begin{equation}\label{abc14}
d_{k+1,p,q}:=(d_{k,p,q})_{p,q}\;\;\mbox{ and }\;\;d_{k+1,s,q}:=(d_{k,s,q})_{s,q}.
\end{equation}
Then by Proposition \ref{fp1}(i) and the inductive assumption, we have for every $x=(x_i),y=(y_i)\in X^\infty_{k+1}$,
$$
d_{k+1,p,q}(x,y)=(d_{k,p,q})_{p,q}(x,y)=\left(\sum_{i\in\N^*}q^id^p_{k,p,q}(x_i,y_i)\right)^{1/p}\leq  \left(\sum_{i\in\N^*}q^i\frac{1}{(1-q)^k}d^p_{k,s,1}(x_i,y_i)\right)^{1/p}$$ $$=\frac{1}{(1-q)^{k/p}}(d_{k,s,1})_{p,q}(x,y)\leq \frac{1}{(1-q)^{k/p}}\frac{1}{(1-q)^{1/p}}(d_{k,s,1})_{s,1}(x,y)=\frac{1}{(1-q)^{(k+1)/p}}d_{k+1,s,1}(x,y),
$$
and similarly
$$
d_{k+1,s,q}(x,y)\leq d_{k+1,s,1}(x,y).
$$
\begin{remark}\label{finalaa1}\emph{
At each step of the above construction we saw that a sequence $(x_i)\subset X^\infty_k$ bounded with respect to $d_{k,s,1}$, is also bounded with respect to $d_{k,p,q}$ and $d_{k,s,q}$, $q<1$. This means that the space $(X^\infty_{k+1},d_{k+1,s,q})$ is a metric subspace of $(\ell_\infty(X^\infty_k,d_{k,s,q}),(d_{k,s,q})_{s,q})$, and similarly, the space $(X^\infty_{k+1},d_{k+1,p,q})$ is a metric subspace of $(\ell_\infty(X^\infty_k,d_{k,p,q}),(d_{k,p,q})_{p,q})$. In fact, if $X$ is unbounded, they are proper subspaces.\\
However, in the case when $(X,d)$ is bounded, then by an easy induction we can see that for all $k\geq 1$, $X^\infty_k=X_k$, i.e, $X^\infty_k$ is exactly the product $\prod_{i=0}^\infty X_{k-1}=\prod_{i=0}^\infty X^\infty_{k-1}$, and hence all these spaces coincide.}
\end{remark}

Now we provide some natural description of considered metrics $d_{k,p,q}$ and $d_{k,s,q}$. We shall start with some notation. 
If $x = (x_0, x_1, ...) \in X_1$ then for any $i\in\N^*$, define 
\begin{equation}\label{abc25}x^{(i)} := x_i \in X_0.
\end{equation}
Assume that for some $k\in\N$, all $x\in X_k$, all $0\leq j\leq k-1$ and all $i_0,...,i_j\in\N^*$, we defined all $x^{(i_0, ..., i_j)}$. 
Then, for any $x = (x_0, x_1, ...) \in X_{k+1}$, any $0\leq j\leq k$ and any $i_0,...,i_j\in\N^*$, we define:
\begin{equation}\label{abc26}\left. \begin{array}{ll} 
x^{(i_0)} := x_{i_0}, \textrm{ if $j=0$}; \\
x^{(i_0, ..., i_j)} := \left( x_{i_0}\right)^{(i_1, ..., i_j)}= \left( x^{(i_0)}\right)^{(i_1, ..., i_j)}, \textrm{ if $j\geq 1$}. 
\end{array} \right.
\end{equation}
{This inductive formula can be understood as taking one by one the coefficient in the ,,nested'' sequence of $x$'s. For example, if $x=((x_0^0,x^0_1,...),(x^1_0,x^1_1,...),...)\in X_2$, then $x^{(2)}=(x^2_0,x^2_1,...)$ and $x^{(2,3)}=x^2_3$, etc.}
The following lemma lists basic properties of spaces $(X^\infty_k,d_{k,p,q})$ and $(X^\infty_k,d_{k,s,q})$.
\begin{lemma}\label{filiplemma1}
In the above frame,
\begin{itemize}
\item[(i)] For any $k\in\N$ and $x,y \in X^{\infty}_k$,
\begin{equation*}\label{11fil22} d_{k,p,q} (x, y) = \left( \sum_{i_0, ..., i_{k-1} \in \N^*} q^{(i_0+...+i_{k-1})} d^p\left(x^{(i_0, ..., i_{k-1})}, y^{(i_0, ..., i_{k-1})}\right) \right)^{1/p},
\end{equation*}
\begin{equation*}\label{11fil2aaa} d_{k,s,q} (x, y) =\sup\{q^{(i_0+...+i_{k-1})} d\left(x^{(i_0, ..., i_{k-1})}, y^{(i_0, ..., i_{k-1})}\right):i_0, ..., i_{k-1} \in \N^*\}.
\end{equation*}
\item[(ii)] If $(X,d)$ is bounded, then for every $k\in\N$, 
$X^\infty_k=X_k.
$
\item[(iii)] If $D\subset X$ is bounded, then for every $k\in\N^*$, $D_k\subset X_k^{\infty}$. Moreover,
\begin{equation*}\label{11fil1}
\on{diam}_{d_{k,p,q}}(D_k)= \left(1-q\right)^{-\frac{k}{p}}\on{diam}_d(D)\;\;\;\;\;\;\;\mbox{and}\;\;\;\;\;\;\;
\on{diam}_{d_{k,s,q}}(D_k)=\on{diam}_{d}(D).     
\end{equation*}
\item[(iv)] If $(X,d)$ is bounded and $q<1$, then for every $k\in\N$, the topology on $X^\infty_k=X_k$ induced by any of metrics $d_{k,p,q}$ and $d_{k,s,q}$ is exactly the Tychonoff product topology. 
\end{itemize}
\end{lemma}
\begin{proof}Part (i) can be easily proved by induction. Part (ii) was already commented in the previous remark. Part (iii) follows from definition, parts (i) and (ii), and the fact (which can be proved by an easy induction) that if $(Y,d)$ is a metric subspace of $(X,d)$, then $(Y^\infty_k,d_{k,s,q})$ is a subspace of $(X^\infty_k,d_{k,s,q})$, and $(Y^\infty_k,d_{k,p,q})$ is a subspace of $(X^\infty_k,d_{k,p,q})$ for all $k\geq 0$.\\
Finally, (iv) follows from (ii),(iii), Proposition \ref{fp1}(iv) and Remark \ref{finalaa1}.
\end{proof}

\section{GIFSs of infinite order and the Hutchinson--Barnsley theorem}

Throughout the section, $(X,d)$ will be a metric space, and $d_{p,q},d_{s,q}$ will be metrics defined as in the Section~\ref{section2}.
\begin{definition}\emph{
A finite family $\F=\{f_1,...,f_n\}$ of maps $\Xinf\to X$ which satisfy (C1) will be called a }generalized iterated function system of infinite order \emph{  (GIFS$_\infty$ in short).\\
As was already observed, every GIFS$_\infty\;\F=\{f_1,...,f_n\}$ generates the map $\F:\KXinf\to \KX$ defined by
$$
\F(K_0,K_1,...):=\overline{f_1\left(\prod_{k=0}^\infty K_k\right)}\cup...\cup \overline{f_n\left(\prod_{k=0}^\infty K_k\right)}.
$$
}
\end{definition}
\begin{remark}\label{remark2}\emph{
By Theorem \ref{Hinf=H}(1), if $f:\Xinf\to X$ is continuous with respect to any of metrics $d_{p,q}$ or $d_{s,q}$ where $q<1$, then for every $(K_k)\in\KXinf$, we have $f(\prod_{k=0}^\infty K_k)\in\K(X)$, that is, $f$~satisfies (C2). In particular, if $\F=\{f_1,...,f_n\}$ consists of such mappings, then}
\begin{equation}\label{fili1}
\F(K_0,K_1,...)={f_1\left(\prod_{k=0}^\infty K_k\right)}\cup...\cup {f_n\left(\prod_{k=0}^\infty K_k\right)}.
\end{equation}
\end{remark}


\begin{definition}\emph{
Let $\F=\{f_1,...,f_n\}$ be a \gifs. We say that $A\in\K(X)$ is an} attractor \emph{or a} fractal generated by $\F$,\emph{ if}
\begin{equation}
\F(A,A,...)=A.
\end{equation}

\end{definition}

We will consider four types of contractive conditions for a \gifs$\;\F=~\{f_1,...,f_n\}$:
\begin{equation*}
(S_1)\;\;\;\;\;\;\;L_{(s,1)}(\F):=\max\{L_{s,1}(f_i):i=1,...,n\}<1\;\;\mbox{and each $f_i$, $i=1,...,n$, satisfy (C1) condition;}
\end{equation*}
\begin{equation*}
(S_2)\;\;\;\;\;\;\;L_{(s,1)}(\F):=\max\{L_{s,1}(f_i):i=1,...,n\}<1\;\;\mbox{and each $f_i$, $i=1,...,n$, satisfy (C2) condition;}
\end{equation*}
\begin{equation*}
\;\;\;\;(Q)\;\;\;\;\;\;\;\;L_{(s,q)}(\F):=\max\{L_{s,q}(f_i):i=1,...,n\}<1\;\;\mbox{for some $q\in(0,1)$;}\;\;\;\;\;\;\;\;\;\;\;\;\;\;\;\;\;\;\;\;\;\;\;\;\;\;\;\;\;\;\;\;\;\;\;\;\;\;\;\;\;\;\;\;\;\;\;\;\;\;\;\;\;\;\;\;
\end{equation*}
\begin{equation*}
\;\;\;\;(P)\;\;\;\;\;\;\;\;L_{(p,q)}(\F):=\max\{L_{p,q}(f_i):i=1,...,n\}<(1-q)^{1/p}\;\;\mbox{for some $q\in(0,1)$ and $p\in[1,\infty)$.}\;\;\;\;\;\;\;\;\;\;\;\;\;\;
\end{equation*}
\begin{remark}\label{abc28}\emph{\\
(1) The recalled Theorem \ref{se2} says that if $\F$ satisfies $(S_1)$ then $\F$ generates a unique fractal set $A_\F$, and for every sequence $(K_k)\in\KXinf$, the sequence $(Y_k)$ defined as in (\ref{abc21}) converges to $A_\F$.\\
(2) By Remark \ref{abc7}, Remark \ref{se4} and Remark \ref{remark2}, we have that:
\begin{equation}\label{abc23}
\;\;\;\;(Q)\;\iff \;(P)
\end{equation}
and each of them implies $(S_2)$. Moreover, $(S_2)$ clearly implies $(S_1)$}. 
\end{remark}

Later we show that none of these implications can be reversed.\\
The next result shows that GIFSs$_\infty$ satisfying (Q) (or, equivalently, (P)) generate fractals which satisfy more restrictive conditions. This is our first main result of the paper and can be considered as a version of the H--B theorem for GIFS$_\infty$:
\begin{theorem}\label{ft1}
Let $(X, d)$ be a complete metric space and $\F=\{f_1,...,f_n\}$ be a GIFS$_\infty$ which satisfies~$(Q)$ (or, equivalently, (P)). 
Then $\F$ generates a unique fractal $A_\F$.\\
Moreover, for every sequence $(K_k)~\in~\KXinf$, the sequence of generalized iterations $(K^k)$ defined by:
\begin{equation}\label{fili2}
K^1:=\F(K_0,K_1,...),
\end{equation}
\begin{equation}\label{fili3}
K^{k+1}:=\F(K^k,...,K^1,K_0,K_1,...)
\end{equation}
converges to $A_\F$ with respect to the Hausdorff-Pompeiu metric. More precisely, for every $(K_k)\in\KXinf$,
\begin{itemize}
\item[-] if $L_{(p,q)}(\F)<(1-q)^{1/p}$, then
$$H^d(K^k, A_\F)\leq L_{(p,q)}(\F)\frac{((L_{(p,q)}(\F))^p+q)^{\frac{k-1}{p}}}{1-((L_{(p,q)}(\F))^p+q)^{\frac{1}{p}}} H^d_{p,q}((K_0,K_1,K_2...),(K^1,K_0,K_1,...));$$
\item[-] if $L_{(s,q)}(\F)<1$, then
$$H^d(K^k, A_\F)\leq L_{(s,q)}(\F)\frac{(\max\{L_{(s,q)}(\F), q\})^{k-1}}{1-\max\{L_{(s,q)}(\F), q\}} H^d_{s,q}((K_0,K_1,K_2...),(K^1,K_0,K_1,...)).$$
\end{itemize}
\end{theorem}
\begin{remark}\emph{Observe that the sequence $(K_k)$ is defined according to (\ref{fili4}) and (\ref{fili5}) for mapping $\F$.}
\end{remark}

We precede the proof with the lemma which lists two known properties of the Hausdorff-Pompeiu metric
\begin{lemma}\label{fl1}
Let $(Y,\rho),(Z,\eta)$ be metric spaces.
\begin{itemize}
\item[i)] If $A_i,B_i,\;i\in I$ are families of subsets of $Y$, then $H^\rho(\bigcup_{i\in I}A_i,\bigcup_{i\in I}B_i)\leq \sup_{i\in I}H^\rho(A_i,B_i)$.
\item[ii)] If $f:Y\to Z$, then for every $A,B\subset Y$, $H^\eta(f(A),f(B))\leq Lip(f)H^\rho(A,B)$.
\end{itemize}
\end{lemma}

\begin{proof}(of Theorem \ref{ft1})
Assume that $L_{(p,q)}(\F):=\max\{L_{p,q}(f_i): i=1,...,n\}<1$.
For every $(K_k), (D_k) \in\KXinf$, we have by Theorem \ref{Hinf=H}, Remark \ref{remark2} and  Lemma \ref{fl1},
$$H^d\left(\F((K_k)\right) , \F((D_k)) =H^d\left(\bigcup_{i=1}^nf_i\left(\prod_{k=0}^\infty K_k\right), \bigcup_{i=1}^nf_i\left(\prod_{k=0}^\infty D_k\right)\right)  $$ $$\leq \max\left\{H^d\left(f_i\left(\prod_{k=0}^\infty K_k\right), f_i\left(\prod_{k=0}^\infty D_k\right)\right):i=1,...,n\right\} \leq \max\left\{L_{p,q}(f_i)H^{d_{p,q}}\left(\prod_{k=0}^\infty K_k,\prod_{k=0}^\infty D_k\right):i=1,...,n\right\}$$
$$
\leq L_{(p,q)}(\F)H^{d_{p,q}}\left(\prod_{k=0}^\infty K_k,\prod_{k=0}^\infty D_k\right)\leq L_{(p,q)}(\F)H^d_{p,q}((K_k),(D_k)).
$$
Hence the Lipschitz constant $L_{p,q}(\F)\leq L_{(p,q)}(\F)<(1-q)^{1/p}$, so the assumptions of Theorem \ref{Banach_inf} are satisfied, and the thesis follows also from this theorem and a fact that $\K(X)$ is complete provided $X$ is complete.\\
Similarly we can prove that the Lipschitz constant $L_{s,q}(\F)\leq L_{(s,q)}(\F)$, hence also the last estimation follows from Theorem \ref{Banach_inf}.
\end{proof}
We are ready to show that the implications stated in Remark \ref{abc28} cannot be reversed. In fact, we will prove that we cannot extend the thesis of Theorem \ref{se2}.

\begin{example}\label{exampleee1}\emph{(1) Let $X=\{\frac{1}{k}:k\in\N\}\cup\{0\}$, and define $f_1,f_2:\ell_\infty(X)\to X$ by
$$
f_1((x_k)):=\frac{1}{2}\max\{x_k:k\in\N^*\},\;\;\;f_2((x_k)):=1.
$$
Secelean proved in \cite[Example 3.1]{Se} that $L_{s,1}(f_1)=\frac{1}{2}$ and $f_1$ satisfies $(C2)$. Clearly, $L_{s,1}(f_2) = 0$ and $f_2$ also satisfies $(C2)$. In particular, the \gifs$\;\F:=\{f_1,f_2\}$ satisfies $(S_2)$, hence it has an attractor $A_\F$. It is easy to see that 
$$A_\F = \{0\} \cup \{2^{-n}: n\in \N^*\}.$$
We will show that the thesis of Theorem \ref{ft1} is not satisfied. Let $K_0=K_1=...=X$, and define
$$
B:=\left\{\frac{1}{2j}:j\in\N\right\}\cup\{0, 1\}.
$$
It is easy to see that for every $k\geq 2,\;
K^k=\F(K^{k-1},...,K^1,K_0,K_1,...)=B.
$ 
\\
Hence the sequence of generalized iterates of $(K_k)$ does not converge to $A_\F$.\\
(2) In fact, even more simple example is the \gifs$\;\F':=\{f_1\}$. Indeed, its attractor is $\{0\}$, but for every $(K_k)\in\KXinf$ so that for some $k_0\in\N^*$, there is $x\in K_{k_0}\setminus\{0\}$, we have that $H(K^k,\{0\})\geq \frac{1}{2}x$ for every $k\in\N$.\\
(3) The last example we want to present here is more natural, but it satisfies just $(S_1)$. Let $X=[0,2]$, $f_1((x_k)):=\frac{1}{2}\sup\{x_k:k\in\N^*\}$ and $f_2((x_k)):=\frac{1}{2}\sup\{x_k:k\in\N^*\}+\frac{1}{4}$. Then $f_1,f_2$ satisfy (C1) and $L_{(s,1)}(\F)=\frac{1}{2}<1$, hence $(S_1)$ is fulfilled, and the set $[0,\frac{1}{2}]$ is the unique fractal of $\F$. On the other hand, if $K_k:=[0,2]$ for $k=0,1,...$, then $K^k=\F(K^{k-1},...,K^1,K_0,K_1,...)=[0,\frac{5}{4}]$ for $k\geq 2$. Thus $(K_k)$ does not converge to $A_\F=[0,\frac{1}{2}]$.\\
As we announced, $f_1,f_2$ do not satisfy $(C2)$ - setting $K_k:=\{0\}\cup[1+\frac{1}{k+1},2]$, we clearly have that $(K_k)\in\ell_\infty(\K(X))$, but $f_1(\prod_{k=0}^\infty K_k)=\{0\}\cup\left(\frac{1}{2},1\right]$ and $f_2(\prod_{k=0}^\infty K_k)=\{\frac{1}{4}\}\cup\left(\frac{3}{4},\frac{5}{4}\right]$. In particular, 
$$\F(K_0,K_1,...)=\left\{0, \frac{1}{4}\right\} \cup \left(\frac{1}{2},\frac{5}{4}\right]$$ and $\F$ does not satisfy $(S_2)$.
}
\end{example}
\begin{remark}\emph{
Let $\F=\{f_1,...,f_n\}$ be a \gifs \ which satisfies (Q) (or, equivalently, (P)). Let $(K_k)\in\KXinf$ be such that each $K_k$ is singleton, and let $(K^k)$ be the sequence of generalized iterates of $\F$. It is easy to see that each set $K^k$ is finite. In fact, by induction it can be shown that
$
\on{card}(K^k)\leq n^{2^{k-1}}.
$ 
Hence we can use also sets $K^k$ to present an image of the fractal (we use it in the next example). Also, by the last part of Theorem \ref{ft1}, we can estimate the distance between $K^k$ and~$A_\F$.\\
On the other hand, the sets $(Y_k)$ defined as in Theorem \ref{se2} will automatically be infinite.}
\end{remark}
\begin{example}\label{ex5}\emph{
Let $\F=\{f_1,f_2,f_3,f_4\}$, where $f_1,f_2,f_3,f_4:\ell_\infty(\R^2)\to \R^2$ are defined by
$$
f_1((x_k,y_k)):=\sum_{k=0}^\infty \frac{1}{10\cdot 4^k}\left(x_k,y_k\right)
\;\;\;\;\;\;\;\;\;\;\;
f_2((x_k,y_k)):=\left(0,\frac{1}{2}\right)+\sum_{k=0}^\infty \frac{1}{10\cdot 4^k}\left(x_k,y_k\right)
$$
$$
f_3((x_k,y_k)):=\left(\frac{1}{2},0\right)+\sum_{k=0}^\infty \frac{1}{10\cdot 4^k}\left(x_k,y_k\right)
\;\;\;\;\;\;\;\;\;\;\;
f_4((x_k,y_k)):=\left(\frac{1}{2},\frac{1}{2}\right)+\sum_{k=0}^\infty \frac{1}{10\cdot 4^k}\left(x_k,y_k\right)
$$
It is easy to see that $L_{(s,\frac{1}{2})}(\F)\leq \frac{1}{5}$ (when considering maximum metric on $\R^2$), hence the assumption~(Q) is satisfied. Set $K_i=\{(0,0)\}, i\in\N^*$. In the following picture we present some first sets of the sequence $(K^k)$, i.e., some first approximations of the attractor $A_\F$ of $\F=\{f_1,f_2,f_3,f_4\}$. 
\begin{center}
\includegraphics{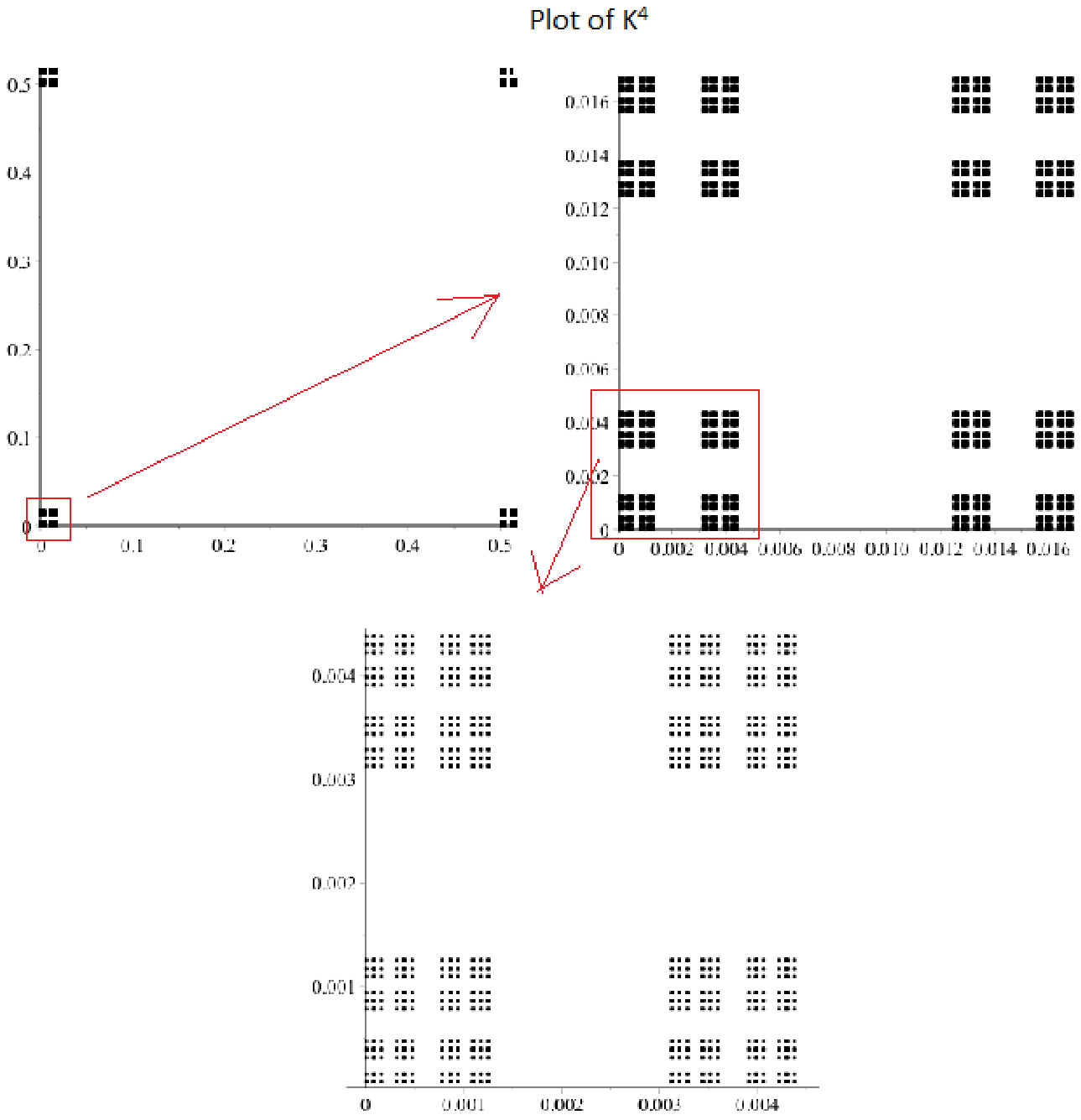}
\end{center}
}\end{example}

As a corollary of Theorem \ref{fp5}, we get that the attractor of a \gifs$\;$ is a limit of attractors of certain GIFSs.
\begin{theorem}\label{abc24}
Let $X$ be complete and $\F=\{f_1,...,f_n\}$ be a \gifs \ satisfying (Q) (or, equivalently, (P)). Choose $x\in X$, and for every $m\in\N$, let $\F_m=\{{f}^m_1,...,{f}^m_n\}$ be a GIFS of order $m$ defined by
$$
{f}^m_i(x_0,...,x_{m-1}):=f_i(x_0,...,x_{m-1},{x},{x},...),\;i=1,...,n.
$$
Then each $\F_m$ satisfies the assumptions of Theorem \ref{gifsinitial}, and the sequence $(A_{\F_m})$ of attractors of $\F_ms$ converges to $A_\F$, the attractor of $\F$.
\end{theorem}
\begin{proof}
Assume (P). By the proof of Theorem \ref{ft1} we see that the Lipschitz constant of the map $\F:\KXinf\to\K(X)$ satisfies (\ref{fp1,1}). Also, for every $K_0,...,K_{m-1}\in\K(X)$,
$$
\F_m(K_0,...,K_{m-1})=\bigcup_{i=1}^n{f}^m_i(K_0\times...\times K_{m-1})=\bigcup_{i=1}^n{f}_i(K_0\times...\times K_{m-1}\times\{x\}\times\{x\}\times...)=$$ $$=\F(K_0,..., K_{m-1},\{x\},\{x\},....).
$$
Hence Theorem \ref{fp5} implies that the Lipschitz constant $\on{Lip}(\F_m)$ of each $\F_m$ is less then one (when considering the maximum metric on $\K(X)^m$), and the sequence of attractors $(A_{\F_m})$ of $\F_ms$ converges to $A_\F$, the attractor of $\F$. Moreover, using Theorem \ref{fp5} for each $f_i$, we also have that each $\F_m$ consists of generalized Banach contractions (when considering the maximum metric on $X^m$). Hence each $\F_m$ satisfies the assumptions of Theorem \ref{gifsinitial}.
\end{proof}
\begin{remark}\emph{
In the paper \cite{JaMS}, together with Jaros, we presented algorithms generating images of GIFSs fractals. Hence, with a help of the above result we can get an image of an approximation of the fractal $A_\F$ - we first choose large enough $m$, define the GIFS $\F_m$ and generate the image of the attractor $A_{\F_m}$.\\
Note that in the formulation of Theorem \ref{fp5} from \cite{JMS} we estimated the speed of the convergence $x_m^*\to x^*$. Also, the speed of (parts of) algorithms presented in \cite{JaMS} were calculated.}
\end{remark}

By Remark \ref{abc8}, thesis of Theorem \ref{abc24} may not hold uder assumption $(S_1)$. We will give a bit less trivial example:
\begin{example}\emph{
Consider the \gifs $\;\F=\{f_1, f_2\}$ from Example \ref{exampleee1}(1). For every $m\in\N$, set
$$
{f}_i^m(x_0,...,x_{m-1}):=f_i(x_0,...,x_{m-1},1,1,...).
$$
Then the set $\{\frac{1}{2},1\}$ is the attractor of each $\F_m=\{f_1^m,{f_2}^m\}$, but, clearly $A_\F\neq \{\frac{1}{2},1\}$. However, $\F$ satisfies $(S_2)$.
}\end{example}

\begin{remark}\emph{
Finally, let us remark that in view of the equivalence $(P)\Leftrightarrow (Q)$, it is enough to develop the theory for GIFSs$_\infty$ satisfying one of this conditions (and, of course, for more general (S)). However, the $"$machinery$"$ works nice for both types of metrics $d_{p,q}$ and $d_{s,q}$ (and, in connection, for both conditions (P) and (Q)). In particular, we get natural estimations in Theorem \ref{ft1}. Hence we will formulate all the results for both cases, but we will give proofs just for the more difficult one.}
\end{remark}

\section{A generalized code space for GIFS$_\infty$}

In this section we will construct and investigate a counterpart of the code space for GIFSs$_\infty$. Recall that in the case of classical IFSs consisting of $n$ maps, the code space is the Cantor space $\mathbf{\Omega}:=\prod_{k=0}^\infty\{1,...,n\}$ with the product topology. Strobin and Swaczyna in \cite{SS2} defined and investigated a counterpart of the code space for GIFSs (see also \cite{M1}). In our construction we will follow the ideas from \cite{SS2}.\\
Let us also note that here we will just consider "abstract" code spaces - the relationships between GIFSs$_\infty$ and their code spaces will be investigated later. 


\subsection{A generalized code space}
Let $n\in\N$. At first, define $\Omega_0,\Omega_1,\Omega_2,...$ according to (\ref{abc0}) for a set $\Omega:=\{1,..., n\}$. 
Now let $d$ be the discrete metric on $\Omega$. Since $(\Omega,d)$ is bounded, Lemma \ref{filiplemma1}(ii) implies that $\Omega_k=\Omega_k^\infty$, and hence we can consider the metrics $d_{k,p,q}$ and $d_{k,s,q}$ defined in (\ref{abc11})-(\ref{abc14}) on $\Omega_k$. 
The following lemma is a straight consequence of Lemma \ref{filiplemma1} and a fact that $\on{diam}_{d}(\Omega)=1$.
\begin{lemma}\label{filiplemma1,5}In the above frame:
\begin{itemize}
\item[(i)] $\Omega_k=\Omega_k^\infty$;
\item[(ii)] if $q<1$, then for every $k\in\N$, $d_{k,p,q}$ and $d_{k,s,q}$ induce the same compact topology on $\Omega_k$ - exactly the Tychonoff product topology;
\item[(iii)] $d_{k,s,1}$ is the discrete topology on $\Omega_k$;
\item[(iv)] 
$
\on{diam}_{d_{k,p,q}}(\Omega_k)=(1-q)^{-k/p}\;\;\;\;\;\;\mbox{and}\;\;\;\;\;\;
\on{diam}_{d_{k,s,q}}(\Omega_k)=1.
$
\end{itemize}
\end{lemma}
Now for every $k\in\N^*$, let 
$$_k\Omega := \Omega_0 \times ... \times \Omega_k = \prod_{i=0}^k \Omega_i$$
and 
$$\Omega_< := \bigcup_{k\in\N^*} {_k\Omega}. $$
Finally, put
$$\mathbf{\Omega}:=\prod_{k=0}^\infty \Omega_k.$$
The space $\mathbf{\Omega}$ will be called \emph{the code space}.\\ 
Now we define certain metrics on the code space $\mathbf{\Omega}$. For every $\alpha~=~(\alpha_0, \alpha_1, ...), \beta=(\beta_0, \beta_1, ...) \in \mathbf{\Omega}$, set
\begin{equation*}d_{(s,q)}(\alpha, \beta) := \sup\left\{q^kd_{k,s,q}(\alpha_k, \beta_k):k\in\N^* \right\},\;\;\mbox{if}\;q\in(0,1],\end{equation*}
\begin{equation*}d_{(p,q)}(\alpha, \beta) := \left( \sum_{k\in\N^*} \left(\frac{1-q}{2}\right)^kd_{k,p,q}^p(\alpha_k, \beta_k) \right)^{1/p},\;\;\mbox{if}\;q\in(0,1),\;p\in[1,\infty).\end{equation*}

The following lemma is straightforward ((ii) and (iii) follow from earlier observations):
\begin{lemma}\label{filiplemma2}In the above frame:
\begin{itemize}
\item[(i)] the functions $d_{(p,q)},d_{(s,q)}$ are metrics;
\item[(ii)] if $q<1$, then $d_{(p,q)}$ and $d_{(s,q)}$ induce the compact topology on $\mathbf{\Omega}$ - exactly the Tychonoff product topogy;
\item[(iii)] the metric $d_{(s,1)}$ is the discrete metric on $\mathbf{\Omega}$.
\end{itemize}
\end{lemma}
From the above lemma, we see that $\mathbf{\Omega}$ is bounded under both types of metrics: $d_{(p,q)}$ or $d_{(s,q)}$. Hence $\ell_\infty(\mathbf{\Omega})=\prod_{i=0}^\infty\mathbf{\Omega}$, no matter which metric $d_{(p,q)}$, $d_{(s,q)}$ for $q<1$, or $d_{(s,1)}$, we consider on $\mathbf{\Omega}$. Thus we will sometimes write $\ell_\infty(\mathbf{\Omega})$ instead of $\prod_{i=0}^\infty \mathbf{\Omega}$.
\subsection{Canonical GIFS$_\infty$ on the code space $\mathbf{\Omega}$}\label{subsection52}
In the case of classical code space $\prod_{k=0}^\infty\{1,...,n\}$ for IFSs, there is considered the special IFS $\{\sigma_1,...,\sigma_n\}$, where each $\sigma_i(\alpha_0,\alpha_1,...):=(i,\alpha_0,\alpha_1,...)$ is the appropriate shift. In \cite{SS2} we introduced a counterpart of this construction for the GIFS's case. Here we will introduce the \gifs's one. The idea is similar to that from \cite[Proposition 2.4]{SS2}.\\
At first, let us introduce some further notations.\\
If $k\geq 1$ and 
$$\alpha = (\alpha_0, \alpha_1, ..., \alpha_k) \in {{_k}\Omega},$$
then for any $i\in\N^*$, we set (recall here (\ref{abc26}))
\begin{equation}\label{abc29}\alpha(i) := \left(\alpha_1^{(i)}, ..., \alpha_k^{(i)}\right).\end{equation}
In other words, if 
$\alpha = (\alpha_0,(\alpha_1^{(0)},\alpha_1^{(1)},...),(\alpha_2^{(0)},\alpha_2^{(1)},...),...,(\alpha_k^{(0)},\alpha_k^{(1)},...))$, then 
$\alpha(0)=(\alpha_1^{(0)},\alpha_2^{(0)},...,\alpha_k^{(0)})$, $\alpha(1)=(\alpha_1^{(1)},\alpha_2^{(1)},...,\alpha_k^{(1)})$ etc.\\
Clearly $\alpha(i)\in {_{k-1}\Omega}$.\\
If $\alpha\in \mathbf{\Omega}$ we define $\alpha(i)\in\mathbf{\Omega}$ in an analogous way, that is, if $\alpha=(\alpha_0,\alpha_1,\alpha_2,...)$, then 
\begin{equation}\label{abc30}
\alpha(i) := \left(\alpha_1^{(i)},\alpha_2^{(i)},...\right).
\end{equation}

Now we will define an announced family of mappings. 
Let $\tau_1, ..., \tau_n: \ell_\infty(\mathbf{\Omega}) \to\mathbf{\Omega}$ be defined as follows: 
if $(\alpha_0,\alpha_1,\alpha_2,...)\in\ell_\infty(\mathbf{\Omega})$ with $\alpha_i~=~(\alpha_i^{(0)}, \alpha_i^{(1)}, \alpha_i^{(2)}, ...)$, then set:
$$\tau_j(\alpha_0, \alpha_1, ...) := \left(j, \left(\alpha_0^{(0)}, \alpha_1^{(0)}, \alpha_2^{(0)}, ...\right), \left(\alpha_0^{(1)}, \alpha_1^{(1)}, \alpha_2^{(1)}, ...\right), ...\right).$$
Finally, define
$$
\F_\mathbf{\Omega}:=\{\tau_1,...,\tau_n\}.
$$
\begin{theorem}\label{filipproposition1}In the above frame, 
\begin{itemize}
\item[(i)]
$
\mathbf{\Omega}=\F_\mathbf{\Omega}(\ell_\infty(\mathbf{\Omega}))=\tau_1(\ell_\infty(\mathbf{\Omega}))\cup...\cup\tau_n(\ell_\infty(\mathbf{\Omega}))$;
\item[(ii)] for every $j=1,...,n$, $L_{p,q}(\tau_j)  =\left(\frac{1-q}{2}\right)^{1/p}$, provided we consider the metric $d_{(p,q)}$ on $\mathbf{\Omega}$;
\item[(iii)] for every $j=1,...,n$, $L_{s,q}(\tau_j)=q$, provided we consider the metric $d_{(s,q)}$ on $\mathbf{\Omega}$.
\end{itemize}
In particular, if we consider any of metrics $d_{(p,q)}$ or $d_{(s,q)}$ for $q<1$ on the code space $\mathbf{\Omega}$, then $\mathbf{\Omega}$ is the fractal generated by the \gifs$\;\F_\mathbf{\Omega}$.
\end{theorem}
\begin{proof}
Ad(i).
Clearly, $$\bigcup_{j\in\{1,...,n\}} \tau_j(\mathbf{\Omega}\times\mathbf{\Omega}\times...) \subset \mathbf{\Omega}.$$ Take any $\alpha = (\alpha_0, \alpha_1, \alpha_2, ...)\in \mathbf{\Omega}$. It can be easily seen that $\alpha = \tau_{\alpha_0}\left(\alpha(0), \alpha(1), \alpha(2), ...\right)$, hence 
$$\alpha \in \tau_{\alpha_0}(\mathbf{\Omega} \times \mathbf{\Omega} \times ...) \subset \bigcup_{j\in\{1, ..., n\}} \tau_j(\mathbf{\Omega}\times \mathbf{\Omega} \times ...).$$
Ad(ii). Now we prove (ii). Take any $j\in\{1,..., n\}$ and, for simplicity, set $\tau := \tau_j$. Let
$$\alpha = (\alpha_0, \alpha_1, \alpha_2, ...), \beta = (\beta_0, \beta_1, \beta_2, ...) \in \ell_\infty(\mathbf{\Omega}).$$
Then, setting $A:=\left(\frac{1-q}{2}\right)$, we have (we denote $\tau(\alpha) = (\tau(\alpha)_0, \tau(\alpha)_1, ...), \tau(\beta) = (\tau(\beta)_0, \tau(\beta)_1, ...)$):

$$d_{(p,q)}(\tau(\alpha), \tau(\beta)) = \left( \sum_{k\in\N^*} A^k{d_{k,p,q}^p(\tau(\alpha)_k, \tau(\beta)_k)} \right)^{1/p} $$
$$= \left( d^p(j,j) + \sum_{k\in\N} A^k{d_{k,p,q}^p(\tau(\alpha)_k, \tau(\beta)_k)} \right)^{1/p} =A^{1/p}\left( \sum_{k\in\N^*} A^k{d_{{k+1},p,q}^p\left(\tau(\alpha)_{k+1}, \tau(\beta)_{k+1}\right)} \right)^{1/p}$$
$$=A^{1/p}  \left(\sum_{k\in\N^*} A^k{d_{{k+1},p,q}^p\left(\left(\alpha_0^{(k)}, \alpha_1^{(k)}, ...\right), \left(\beta_0^{(k)}, \beta_1^{(k)}, ...\right)\right)} \right)^{1/p} $$
$$=  A^{1/p}\left(\sum_{k\in\N^*} A^k{\sum_{i\in\N^*} q^i d_{k,p,q}^p\left(\alpha_i^{(k)}, \beta_i^{(k)}\right)} \right)^{1/p} = A^{1/p}\left( \sum_{i\in\N^*} q^i \sum_{k\in\N^*}  A^k{d_{k,p,q}^p\left(\alpha_i^{(k)}, \beta_i^{(k)}\right)} \right)^{1/p} $$
$$ = A^{1/p}\left( \sum_{i\in\N^*} q^i d_{(p,q)}^p(\alpha_i, \beta_i) \right)^{1/p} = A^{1/p}  (d_{(p,q)})_{p,q} (\alpha, \beta)=\left(\frac{1-q}{2}\right)^{1/p}(d_{(p,q)})_{p,q} (\alpha, \beta).$$
Hence we get (ii).\\
In a similar way, we can show
$
d_{(s,q)}(\tau(\alpha),\tau(\beta))= q(d_{(s,q)})_{s,q}(\alpha,\beta),
$
which gives (iii).
\end{proof}
\begin{remark}\emph{By (iii), we see that if $n>1$ and we consider $d_{(s,1)}$ metric on $\mathbf{\Omega}$, then $L_{s,1}(\tau_j)=1$, so the assumptions of Theorem \ref{se1} are not satisfied. In fact, $\F_\mathbf{\Omega}$ does not have an attractor in such case. Indeed, as $d_{(s,1)}$ is the discrete metric, an attractor of $\F_\mathbf{\Omega}$ would be finite (as a compact set). But then it would be also the attractor of $\F_\mathbf{\Omega}$, when considering the metric $d_{(p,q)}$ or $d_{(s,q)}$ for $q<1$.
}\end{remark}

\section{Generalized code space and GIFSs of order infinity}
In this section we assume that we work with some fixed \gifs$\;\F=\{f_1, ..., f_n\}$ on a complete metric space $(X, d)$ and $\Omega_k, {_k\Omega}, \Omega_<, \mathbf{\Omega}$ and $X_k$, $X^{\infty}_k$
keep their meaning from the previous sections (in particular, we recall Subsection 3.2). 
\subsection{Counterpart of composition operation -- families $\F_k$} In this subsection we will derive a counterpart of composition of functions. In classical IFS case, if $\{g_1,...,g_n\}$ is an IFS, then we can consider the compositions $g_{\alpha_0}\circ...\circ g_{\alpha_k}$, where $\alpha_0,...,\alpha_k\in\{1,...,n\}$. Note that in such case, the Lipschitz constant $$\on{Lip}(g_{\alpha_0}\circ...\circ g_{\alpha_k})\leq L^{k+1},$$
where $L:=\max\{\on{Lip}(g_i):i=1,...,n\}$ and in consequence, for every bounded set $D$,
$$
\on{diam}(g_{\alpha_0}\circ...\circ g_{\alpha_k}(D))\leq L^{k+1}\on{diam}(D).
$$

In \cite[Section 3]{SS2} we defined the counterpart of composition for GIFSs. Here we bring the ideas from \cite{SS2} to the ''next level''.\\
We will assume here a bit less than $(S_1)$. Namely, we will assume:
\begin{equation}\label{abc15}
L_{(s,1)}(\F):=\max\{L_{s,1}(f_i):i=1,...,n\}<\infty\;\;\;\mbox{and each $f_i$ satisfies (C1) condition.}
\end{equation}
\begin{remark}\label{abc16}\emph{
From Proposition \ref{fp1}(ii),(iii) and Remark~\ref{remark2}, we see that (\ref{abc15}) 
is satisfied provided one of the following holds:
\begin{equation}
L_{(s,q)}(\F):=\max\{L_{s,q}(f_i):i=1,...,n\}<\infty,\;\;\mbox{for some}\;q\in(0,1);
\end{equation}
\begin{equation}
L_{(p,q)}(\F):=\max\{L_{p,q}(f_i):i=1,...,n\}<\infty,\;\;\mbox{for some}\;q\in(0,1),\;p\in[1,\infty).
\end{equation}
}\end{remark}
We will introduce certain families of maps $\F_k=\{f_\alpha:X^\infty_{k+1}\to X:\alpha\in{_k}\Omega\}$ for $k\in\N^*$. The definition will be inductive.\\
For $k=0$, we set
$$
\F_0:=\{f_1,...,f_n\}.
$$
Assume that for some $k\in\N^*$, the family $\F_k$ is already defined.
For every $\alpha=(\alpha_0, \alpha_1, ..., \alpha_k, \alpha_{k+1}) \in {_{k+1}\Omega}$ set
\begin{equation}\label{abcde1}f_\alpha(x_0, x_1, ...) := f_{\alpha_0}\left(f_{\alpha(0)}(x_0), f_{\alpha(1)}(x_1), ...\right),\end{equation}
where $(x_0, x_1, ...) \in X^\infty_{k+2}$. Finally, put 
$$\F_{k+1}=\{f_\alpha:\alpha\in{_{k+1}}\Omega\}.$$

A question arises if the functions $f_\alpha$ are well defined. If $X$ is bounded, then there are no problems since each $X^\infty_k=X_k$. However, in general case it should be justified that if $\alpha\in {_{k}}\Omega$ and $(x_0,x_1,...)\in X^{\infty}_{k+1}$, then the sequence $\left(f_{\alpha(0)}(x_0), f_{\alpha(1)}(x_1), ...\right)\in \Xinf$. 
 We will prove a bit stronger assertion (in the formulation we assume that $\infty^{k+1}=\infty$):
\begin{proposition}\label{abc5}In the above frame, for every $k\in\N^*$:
\begin{itemize}
\item[(i)] the functions $f_\alpha$, $\alpha\in{_k}\Omega$ are well defined;
\item[(ii)] for every $\alpha\in{_k}\Omega$, $\on{Lip}_{d_{k+1,p,q}}(f_\alpha)\leq L_{(p,q)}(\F)^{k+1}$ and $\on{Lip}_{d_{k+1,s,q}}(f_\alpha)\leq L_{(s,q)}(\F)^{k+1}$;
\item[(iii)] if $L_{(p,q)}(\F)<\infty$ and $B\subset X^\infty_{k+1}$ is bounded with respect to $d_{k+1,p,q}$, then  $\{f_{\alpha}(x):\alpha\in{_k}\Omega,\;x\in B\}$ is bounded in $X$;
\item[(iv)] if $L_{(s,q)}(\F)<\infty$ and $B\subset X^\infty_{k+1}$ is bounded with respect to $d_{k+1,s,q}$, then $\{f_{\alpha}(x):\alpha\in{_k}\Omega,\;x\in B\}$ is bounded in $X$.
\end{itemize}
\end{proposition}
\begin{proof}
Assume first that $L_{(p,q)}(\F)<\infty$.\\
The proof will be inductive.
First let $k=0$. Point (i) holds by definition of $\F_0$. Since $L_{(p,q)}(\F)<\infty$ by assumption, also $(ii)$ holds. Now let $B\subset X^{\infty}_1=\ell_\infty(X)$ be $d_{1,p,q}$-bounded, and fix $x_0\in B$ and $\alpha\in{_0}\Omega$. For every $y\in B$ and $\beta\in{_0}\Omega=\{1,...,n\}$, we have 
$$d\left(f_{\alpha}(x_0), f_{\beta}(y)\right) \leq d\left(f_{\alpha}(x_0), f_{\beta}(x_0)\right) + d\left(f_{\beta}(x_0), f_{\beta}(y) \right)  $$
$$\leq \max\left\{d\left(f_ {i}(x_0), f_j(x_0)\right): i,j=1,...,n\right\} + L_{(p,q)}(\F)d_{1,p,q}(x_0, y)  $$
$$\leq \max\left\{d\left(f_ {i}(x_0), f_j(x_0)\right): i,j=1,...,n\right\}  +L_{(p,q)}(\F) \on{diam}_{d_{1,p,q}}(B).$$
Hence we get (iii) and the whole assertion for $k=0$ (as we consider now the case $L_{(p,q)}(\F)<\infty$). Now assume that our claims are true for some $k\in\N^*$.
Let $\alpha=(\alpha_0,\alpha_1,...,\alpha_{k+1})\in{_{k+1}}\Omega$ and $x=(x_0,x_1,...)\in X_{k+2}^\infty$. By the inductive assumption, the sequence $(f_{\alpha(0)}(x_0),f_{\alpha(1)}(x_1),...)\in\Xinf$ (we use the assumption for $B:=\{x_i:i\in\N^*\}$, which is $d_{k+1,p,q}$-bounded, compare Remark~\ref{finalaa1}). Hence $f_\alpha(x_0,x_1,...)$ is well defined. Now if  $x=(x_0,x_1,...),y=(y_0,y_1,...)\in X^{\infty}_{k+2}$ then, by the inductive assumption, we have
$$
d(f_\alpha(x),f_\alpha(y))=d(f_{\alpha_0}(f_{\alpha(0)}(x_0),f_{\alpha(1)}(x_1),...),f_{\alpha_0}(f_{\alpha(0)}(y_0),f_{\alpha(1)}(y_1),...)) $$ 
$$\leq L_{p,q}(f_{\alpha_0}) d_{p,q}((f_{\alpha(0)}(x_0),f_{\alpha(1)}(x_1),...),(f_{\alpha(0)}(y_0),f_{\alpha(1)}(y_1),...))
$$
$$\leq L_{(p,q)}(\F)\left(\sum_{i\in\N^*}q^id^p(f_{\alpha(i)}(x_i),f_{\alpha(i)}(y_i))\right)^{1/p}
$$
$$
\leq L_{(p,q)}(\F)L_{(p,q)}(\F)^{k+1}\left(\sum_{i\in\N^*}q^id_{k+1,p,q}^p(x_i,y_i)\right)^{1/p}=L_{(p,q)}(\F)^{k+2}d_{k+2,p,q}(x,y).
$$
Hence $\on{Lip}_{d_{k+2,p,q}}(f_\alpha)\leq L_{(p,q)}(\F)^{k+2}$. Now let $B\subset X^{\infty}_{k+2}$ be bounded with respect to $d_{k+2,p,q}$ and fix $\tilde{x}=(\tilde{x}_0,\tilde{x}_1,...)\in B$ and $\alpha=(\alpha_0,\alpha_1,...,\alpha_{k+1})\in{_{k+1}}\Omega$.\\
By the inductive assumption, there exists $K<\infty$ such that $d(f_\gamma(\tilde{x}_i),f_\eta(\tilde{x}_j))<K$ for every $i,j\in\N^*$ and $\gamma,\eta\in{_k}\Omega$. Also, set
$$
D:=\max\left\{d\left(f_i(f_{\alpha(0)}(\tilde{x}_0),f_{\alpha(1)}(\tilde{x}_1),...),f_j(f_{\alpha(0)}(\tilde{x}_0),f_{\alpha(1)}(\tilde{x}_1),...)\right):i,j=1,...,n\right\}.
$$
For every $y=(y_0,y_1,...)\in B$ and $\beta=(\beta_0,\beta_1,...,\beta_{k+1})\in{_{k+1}}\Omega$, we have
\begin{equation}\label{11fil3}
d(f_\alpha(\tilde{x}),f_\beta(y))\leq d(f_\alpha(\tilde{x}),f_\beta(\tilde{x}))+d(f_\beta(\tilde{x}),f_\beta(y))\leq
\end{equation}
\begin{equation*}\leq d(f_\alpha(\tilde{x}),f_\beta(\tilde{x}))+L_{(p,q)}(\F)^{k+2}d_{k+2,p,q}(\tilde{x},y)\leq d(f_\alpha(\tilde{x}),f_\beta(\tilde{x}))+L_{(p,q)}(\F)^{k+2}\on{diam}_{d_{k+2,p,q}}(B).
\end{equation*}
We also have
$$
d(f_\alpha(\tilde{x}),f_\beta(\tilde{x}))=d(f_{\alpha_0}(f_{\alpha(0)}(\tilde{x}_0),f_{\alpha(1)}(\tilde{x}_1),...),f_{\beta_0}(f_{\beta(0)}(\tilde{x}_0),f_{\beta(1)}(\tilde{x}_1),...)) $$ 
$$
\leq d(f_{\alpha_0}(f_{\alpha(0)}(\tilde{x}_0),f_{\alpha(1)}(\tilde{x}_1),...),f_{\beta_0}(f_{\alpha(0)}(\tilde{x}_0),f_{\alpha(1)}(\tilde{x}_1),...))+$$ $$+d(f_{\beta_0}(f_{\alpha(0)}(\tilde{x}_0),f_{\alpha(1)}(\tilde{x}_1),...),f_{\beta_0}(f_{\beta(0)}(\tilde{x}_0),f_{\beta(1)}(\tilde{x}_1),...))
$$
$$
\leq\max\{d(f_i(f_{\alpha(0)}(\tilde{x}_0),f_{\alpha(1)}(\tilde{x}_1),...),f_j(f_{\alpha(0)}(\tilde{x}_0),f_{\alpha(1)}(\tilde{x}_1),...)):i,j=1,...,n\}+
$$
$$
+L_{(p,q)}(\F)\left(\sum_{i\in\N^*}q^id^p\left(f_{\alpha(i)}(\tilde{x}_i),f_{\beta(i)}(\tilde{x}_i)\right)\right)^{1/p}\leq
D+L_{(p,q)}(\F)\left(\sum_{i\in \N^*}q^iK^p\right)^{1/p}.
$$
All in all,
$$
d(f_\alpha(\tilde{x}),f_\beta(y))\leq D+L_{(p,q)}(\F)K(1-q)^{-1/p}+L_{(p,q)}(\F)^{k+2}\on{diam}_{d_{k+2,p,q}}(B).
$$
Now since the values $D$ and $K$ depend only on $\alpha$ and $\tilde{x}$, we get that the set $\{f_\beta(y):y\in B,\beta\in{_{k+1}}\Omega\}$ is bounded.\\
This gives the assertion for $k+1$.\\
Now we can proceed in a very similar way, but under the assumption $L_{(s,q)}(\F)<\infty$.
As a consequence, when $q=1$, we get the whole point (i) (as we generally assume that $L_{(s,1)}(\F)<\infty$).\\
Now observe that if in the point (ii) $L_{(p,q)}(\F)=\infty$ (or $L_{(s,q)}(\F)=\infty$), then the inequality $\on{Lip}_{d_{k+1,p,q}}(f_\alpha)\leq L_{(p,q)}(\F)^{k+1}$ (or $\on{Lip}_{d_{k+1,s,q}}(f_\alpha)\leq L_{(s,q)}(\F)^{k+1}$) obviously holds. This ends the proof.

\end{proof}
As a corollary of Lemma \ref{filiplemma1}(iii) and Proposition \ref{abc5}(ii), we get:
\begin{lemma}\label{mniejsze srednice}In the above frame,
let $D\subset X$ be bounded, and let $(D_k)$ be defined according to (\ref{abc0}). 
For any $k\in\N^*$ and $\alpha\in{_k}\Omega$,
\begin{itemize}
\item[(i)] $\on{diam}_d(f_\alpha(D_{k+1}))\leq \left(\frac{L_{(p,q)}(\F)}{(1-q)^{1/p}}\right)^{k+1}\on{diam}_d(D)$;
\item[(ii)] $\on{diam}_{d}(f_\alpha(D_{k+1}))\leq L_{(s,q)}(\F)^{k+1}\on{diam}_d(D)$.
\end{itemize}
In particular, if $L_{(s,1)}(\F)<1$ (for example, if $\F$ satisfies (Q) or (P)), then we have 
\begin{equation}\label{abc166}
\lim_{k\to\infty}\sup\{\on{diam}_d(f_{\alpha}(D_{k+1})):\alpha\in{_k}\Omega\}=0
\end{equation}
and, consequently, for any $\alpha \in \mathbf{\Omega}$, $\on{diam}_d(f_{\alpha|_k}(D_{k+1})) \to 0$, where $\alpha|_k$ denotes the restriction of $\alpha$ to the first $k+1$ elements.\end{lemma}
\begin{remark}\emph{In view of the above considerations, the family $\F_k$ can be considered as a counterpart of the family of all compositions $\{g_{\alpha_0}\circ...\circ g_{\alpha_k}:(\alpha_0,...,\alpha_k)\in\{1,...,n\}^{k+1}\}$ in the classical case, and appropriate families from \cite{SS2}.}
\end{remark}
\subsection{The division of the fractal set}
Recall that in the classical case, if $\{g_1,...,g_n\}$ is an IFS consisting of Banach contractions and $A$ is its fractal, then:
\begin{itemize}
\item[(i)] for every $k\in\N^*$ and $(\alpha_0,...,\alpha_k)\in\{1,...,n\}^{k+1}$, 
$$g_{\alpha_0}\circ...\circ g_{\alpha_k}(A)=\bigcup_{\alpha\in\{1,...,n\}}g_{\alpha_0}\circ...\circ g_{\alpha_k}\circ g_{\alpha}(A);$$
\item[(ii)] for every $k\in\N^*$, $A=\bigcup_{(\alpha_0,...,\alpha_k)\in\{1,...,n\}^{k+1}}g_{\alpha_0}\circ...\circ g_{\alpha_k}(A)$;
\item[(iii)] if $k\in\N^*$ and $(\alpha_0,...,\alpha_k)\in\{1,...,n\}^{k+1}$, then $\on{diam}(g_{\alpha_0}\circ...\circ g_{\alpha_k}(A))\leq L^{k+1}\on{diam}(A)$, where $L~:=~\max\{Lip(g_i):i=1,...,n\}$.
\end{itemize}
This gives the natural division of the fractal $A$ into smaller and smaller pieces. In \cite{SS2} we constructed a counterpart of this division for GIFSs (see \cite[Proposition 3.3]{SS2}). Here we will do it for GIFSs$_\infty$.\\

In this section we assume that a \gifs $\;\F=\{f_1,...,f_n\}$ satisfies $(S_1)$, that is, $L_{(s,1)}(\F)<1$ and each $f_i$ satisfies (C1), and $A_\F$ is its attractor.\\

At first, let $(A_k)$ be the family of sets defined according to (\ref{abc0}), for a set $$A:=A_\F.$$ 
The following lemma is a straightforward consequence of Lemma \ref{filiplemma1} and a fact that $A_\F$ is compact (hence bounded).
\begin{lemma}\label{filiplemma3}In the above frame:
\begin{itemize}
\item[(i)] for every $k\in\N^*$, $A_k\subset X_k^{\infty}$;
\item[(ii)] if $q<1$ and $k\in\N$, then $d_{k,p,q}$ and $d_{k,s,q}$ induce the same compact topology on $A_k$ - exactly the Tychonoff product topology;
\item[(iii)] 
$
\on{diam}_{d_{k,p,q}}(A_k)=\left(1-q\right)^{-\frac{k}{p}}\on{diam}_d(A_\F)\;\;\;\;\mbox{and}\;\;
\on{diam}_{d_{k,s,q}}(A_k)=\on{diam}_d(A_\F).
$
\end{itemize}
\end{lemma}

Now for every $k\in\N^*$ and every $\alpha\in {_k\Omega}$, define
$$A_\alpha := \overline{f_\alpha(A_{k+1})}.$$
It turns out that in the case when $\F$ satisfies $(S_2)$, we can get rid of the closure.
\begin{lemma}\label{juzniemampomyslu}
If $\F$ satisfies $(S_2)$ (that is, $f_1,...,f_n$ additionally satisfy (C2)), then for every $k\in\N^*$ and every $\alpha\in {_k\Omega}$, 
$$A_\alpha = f_\alpha(A_{k+1}).$$
\end{lemma}
\begin{proof}
We just have to use (\ref{abcde1}) and simple inductive argument.
\end{proof}

By definition, $A_{\F} = \overline{f_1(A_{\F} \times A_{\F} \times ...)} \cup ... \cup \overline{f_n(A_{\F}\times A_{\F} \times ...)} = A_{(1)} \cup ... \cup A_{(n)}$. 
It turns out that the following holds (the symbol $\alpha\hat{\;}\beta$ denotes the concatenation of a sequence $\alpha=(\alpha_0,...,\alpha_k)$ with $\beta$, that is $\alpha\hat{\;}\beta=(\alpha_0,...,\alpha_k,\beta)$):
\begin{theorem}\label{kafelkowanie}
In the above frame, let $k\in\N^*$.
\begin{itemize}
\item[(i)] For every $\alpha\in{_k}\Omega$, $$\on{diam}_d(A_\alpha)\leq \left(\frac{L_{(p,q)}(\F)}{(1-q)^{1/p}}\right)^{k+1}\on{diam}_d(A_\F)\;\;\;\;\mbox{ and }\;\;\;\;\on{diam}_d(A_\alpha)\leq L_{(s,q)}(\F)^{k+1}\on{diam}_d(A_\F).$$
\item[(ii)] For every $\alpha\in {_k}\Omega$, 
\begin{itemize}
\item[(iia)] $A_\alpha = \overline{\bigcup_{\beta\in \Omega{_{k+1}}} A_{\alpha \hat{\;} \beta}}$;
\item[(iib)] $A_\F=\overline{\bigcup_{\alpha\in{_k}\Omega}A_\alpha}$.
\end{itemize}
\item[(iii)] If additionally $\F$ satisfies $(S_2)$, then for every $\alpha\in {_k}\Omega$,
\begin{itemize}
\item[(iiia)] $A_\alpha=f_\alpha(A_{k+1})$;
\item[(iiib)] $A_\alpha = \bigcup_{\beta\in \Omega{_{k+1}}} A_{\alpha \hat{\;} \beta}$;
\item[(iiic)] $A_\F=\bigcup_{\alpha\in{_k}\Omega}A_\alpha$.
\end{itemize}

\end{itemize}
\end{theorem}
\begin{remark}\emph{
The point (iii) shows another advantage of the assumption $(S_2)$ - in the definition of $A_\alpha$, and in appropriate divisions we do not need to take any closures.}
\end{remark}
Before we give a proof, we state a lemma:
\begin{lemma}
Assume that $B_0,B_1,...$ are subsets of $X$ such that $\bigcup_{i=0}^\infty B_i$ is bounded. Then 
$$\prod_{i=0}\overline{B_i}=\overline{\prod_{i=0}^\infty B_i},$$
where the last closure is taken with respect to $d_{p,q}$ ($q<1$) or $d_{s,q}$ ($q\leq 1$).
\end{lemma}
\begin{proof}
For any $x=(x_i)$, we have:
$$
(x_i)\in \overline{\prod_{i=0}^\infty{B_i}}\;\;\Leftrightarrow\;\;\forall_{x\in V,\;V-\mbox{open}}\;\; V\cap \prod_{i=0}^\infty{B_i}\neq\emptyset$$ $$\Leftrightarrow\;\;\forall_{i\in\N^*}\;\forall_{x_i\in V_i,\;V_i-\mbox{open in }X}\;V_i\cap B_i\neq\emptyset\;\;\Leftrightarrow\;\;\forall_{i\in\N^*}\;x_i\in \overline{B_i}\;\;\Leftrightarrow\;\;(x_i)\in\prod^\infty_{i=0}\overline{B_i}.
$$
\end{proof}
It is worth to note that in the case when $q<1$, the result is well known as the topology on appropriate product is exactly the Tychonoff product topology. 
\begin{proof}(of Theorem \ref{kafelkowanie})
To prove (i), we just have to use Lemma \ref{mniejsze srednice}(i),(ii) (we take the closure, but it does not change the diameter).\\
Now we prove (ii). We first show that for every $k\in\N^*$ and $\alpha\in{_k}\Omega$,
\begin{equation}\label{zz1}
\overline{\bigcup_{\beta\in\Omega_{k+1}}f_{\alpha\hat{\;} \beta}(A_{k+2})}=\overline{f_\alpha(A_{k+1}).}
\end{equation}
Let $k=0$ and take $i\in {_0}\Omega=\{1,...,n\}$. On one hand by continuity of $f_i, i=1,...,n$, we have
$$
f_i(A_1)=f_i(A_\F\times A_\F\times...)=f_i\left(\bigcup_{\beta_0\in_0\Omega}\overline{f_{\beta_0}(A_1)}\times \bigcup_{\beta_1\in_0\Omega} \overline{f_{\beta_1}(A_1)}\times...\right)$$
$$=\bigcup_{\beta_0\in _0\Omega}\bigcup_{\beta_1\in_0\Omega}...\;f_i\left(\overline{f_{\beta_0}(A_1)}\times \overline{f_{\beta_1}(A_1)}\times...\right)=\bigcup_{(\beta_0,\beta_1,...)\in \Omega_1}f_i\left(\overline{f_{\beta_0}(A_1)\times f_{\beta_1}(A_1)\times...}\right)
$$
$$
\subset \bigcup_{(\beta_0,\beta_1,...)\in \Omega_1}\overline{f_i\left(f_{\beta_0}(A_1)\times f_{\beta_1}(A_1)\times...\right)}=\bigcup_{\beta\in\Omega_1}\overline{f_{(i,\beta)}(A_2)},
$$
which gives us 
$$
\overline{f_i(A_1)}\subset \overline{\bigcup_{\beta\in \Omega_1}\overline{f_{(i,\beta)}(A_2)}}=\overline{\bigcup_{\beta\in \Omega_1}{f_{(i,\beta)}(A_2)}}.
$$
On the other hand,
$$
\bigcup_{\beta\in\Omega_1}f_{(i,\beta)}(A_2)=\bigcup_{(\beta_0,\beta_1,...)\in\Omega_1}f_i(f_{\beta_0}(A_1)\times f_{\beta_1}(A_1)\times....)$$ $$=f_i\left(\bigcup_{\beta_0\in_0\Omega}f_{\beta_0}(A_1)\times\bigcup_{\beta_1\in_0\Omega}f_{\beta_1}(A_1)\times...\right)\subset f_i\left(A_\F\times A_\F\times...\right)=f_i(A_1),
$$
whence 
$$
\overline{\bigcup_{\beta\in\Omega_1}f_{(i,\beta)}(A_2)} \subset \overline{f_i(A_1)}.
$$
In particular, we get (\ref{zz1}) for $k=0$. Now assume that it holds for some $k\in\N^*$, and choose $\alpha\in {_{k+1}}\Omega$. On one hand, we have
$$
f_\alpha(A_{k+2})=f_{\alpha_0}\left(f_{\alpha(0)}(A_{k+1})\times f_{\alpha(1)}(A_{k+1})\times...\right)=f_{\alpha_0}\left(f_{(\alpha_1^{(0)},...,\alpha_{k+1}^{(0)})}(A_{k+1})\times f_{(\alpha_1^{(1)},...,\alpha_{k+1}^{(1)})}(A_{k+1})\times...\right)$$
$$
\subset f_{\alpha_0}\left(\overline{f_{(\alpha_1^{(0)},...,\alpha_{k+1}^{(0)})}(A_{k+1})}\times\overline{ f_{(\alpha_1^{(1)},...,\alpha_{k+1}^{(1)})}(A_{k+1})}\times...\right)
$$ 
$$=f_{\alpha_0}\left(\overline{\bigcup_{\beta_0\in\Omega_{k+1}}f_{(\alpha_1^{(0)},...,\alpha_{k+1}^{(0)},\beta_0)}(A_{k+2})}\times\overline{\bigcup_{\beta_1\in\Omega_{k+1}} f_{(\alpha_1^{(1)},...,\alpha_{k+1}^{(1)},\beta_1)}(A_{k+2})}\times...\right)
$$
$$
\subset \overline{f_{\alpha_0}\left(\bigcup_{\beta_0\in\Omega_{k+1}}f_{(\alpha_1^{(0)},...,\alpha_{k+1}^{(0)},\beta_0)}(A_{k+2})\times\bigcup_{\beta_1\in\Omega_{k+1}} f_{(\alpha_1^{(1)},...,\alpha_{k+1}^{(1)},\beta_1)}(A_{k+2})\times...\right)}
$$
$$
=\overline{\bigcup_{(\beta_0,\beta_1,...)\in\Omega_{k+2}}f_{\alpha_0}\left(f_{(\alpha\hat{\;} \beta)(0)}(A_{k+2})\times f_{(\alpha\hat{\;} \beta)(1)}(A_{k+2})\times...\right)}=\overline{\bigcup_{\beta\in\Omega_{k+2}}f_{\alpha\hat{\;} \beta}(A_{k+3})},
$$
so 
$$\overline{f_\alpha(A_{k+2})}\subset \overline{\bigcup_{\beta\in\Omega_{k+2}}f_{\alpha\hat{\;} \beta}(A_{k+3})}.$$
On the other hand,
$$
\bigcup_{\beta\in\Omega_{k+2}}f_{\alpha\hat{\;} \beta}(A_{k+3})=f_{\alpha_0}\left(\bigcup_{\beta_0\in\Omega_{k+1}}f_{(\alpha_1^{(0)},...,\alpha_{k+1}^{(0)},\beta_0)}(A_{k+2})\times \bigcup_{\beta_1\in\Omega_{k+1}}f_{(\alpha_1^{(1)},...,\alpha_{k+1}^{(1)},\beta_1)}(A_{k+2})\times...\right)$$
$$
\subset f_{\alpha_0}\left(\overline{\bigcup_{\beta_0\in\Omega_{k+1}}f_{(\alpha_1^{(0)},...,\alpha_{k+1}^{(0)},\beta_0)}(A_{k+2})}\times \overline{\bigcup_{\beta_1\in\Omega_{k+1}}f_{(\alpha_1^{(1)},...,\alpha_{k+1}^{(1)},\beta_1)}(A_{k+2})}\times...\right)
$$
$$
=f_{\alpha_0}\left(\overline{f_{(\alpha_1^{(0)},...,\alpha_{k+1}^{(0)})}(A_{k+1})}\times \overline{f_{(\alpha_1^{(1)},...,\alpha_{k+1}^{(1)})}(A_{k+1})}\times...\right)$$ $$\subset \overline{f_{\alpha_0}\left({f_{(\alpha_1^{(0)},...,\alpha_{k+1}^{(0)})}(A_{k+1})}\times {f_{(\alpha_1^{(1)},...,\alpha_{k+1}^{(1)})}(A_{k+1})}\times...\right)}=\overline{f_\alpha(A_{k+2})},
$$
whence 
$$
\overline{\bigcup_{\beta\in\Omega_{k+2}}f_{\alpha\hat{\;} \beta}(A_{k+3})} \subset \overline{f_\alpha(A_{k+2})}.
$$
In particular, we get (\ref{zz1}) for $k+1$, which ends the proof of (\ref{zz1}).\\
We are ready to prove (ii). Let $\alpha \in {_k}\Omega$, then by (\ref{zz1}), we have
$$
A_\alpha=\overline{f_\alpha(A_{k+1})}=\overline{\bigcup_{\beta\in\Omega_{k+1}}f_{\alpha\hat{\;} \beta}(A_{k+2})}=\overline{\bigcup_{\beta\in\Omega_{k+1}}\overline{f_{\alpha\hat{\;} \beta}(A_{k+2})}}=\overline{\bigcup_{\beta\in\Omega_{k+2}}A_{\alpha\hat{\;} \beta}}
$$
hence we get (iia). We show (iib) by induction. If $k=0$, then it clearly holds (even without closures). If it holds for some $k$, then by (iia),
$$
A_\F=\overline{\bigcup_{\alpha\in_k\Omega}A_\alpha}=\overline{\bigcup_{\alpha\in_k\Omega}\overline{\bigcup_{\beta\in\Omega_{k+1}}A_{\alpha\hat{\;} \beta}}}=\overline{\bigcup_{\alpha\in_k\Omega}\bigcup_{\beta\in\Omega_{k+1}}A_{\alpha\hat{\;} \beta}}=\overline{\bigcup_{\alpha\in_{k+1}\Omega}A_\alpha}.
$$
Thus (ii) holds.\\
The point (iiia) was already observed in previous lemma. Using it, we can follow the same lines as in~(ii), but without closures.
\end{proof}
\subsection{The canonical map between the code space and an appropriate \gifs}
The mentioned results concerning IFSs allow to define a natural map between the code space $\prod_{k=0}^\infty\{1,...,n\}$ and an IFS $\{g_1,...,g_n\}$ with the attractor $A$. Namely, for every $\alpha=(\alpha_0,\alpha_1,...)\in\prod_{k=0}^\infty\{1,...,n\}$, let $\pi(\alpha)$ be the unique element of the intersection $\bigcap_{k=0}^\infty g_{\alpha_0}\circ...\circ g_{\alpha_k}(A)$. Then the map $\pi:\prod_{k=0}^\infty\{1,...,n\}\to X$ has the following properties:
\begin{itemize}
\item[(i)] $\pi$ is continuous;
\item[(ii)] $\pi(\prod_{k=0}^\infty\{1,...,n\})=A$; 
\item[(iii)] for every $\alpha\in\prod_{k=0}^\infty\{1,...,n\}$ and every nonempty, closed and bounded $D \subset X$, the sequence\\ $g_{\alpha_0}\circ...\circ g_{\alpha_k}(D)\to \{\pi(\alpha)\}$ with respect to the Hausdorff-Pompeiu metric;
\item[(iv)] for every $\alpha\in\prod_{k=0}^\infty\{1,...,n\}$ and every $x\in X$, the sequence $g_{\alpha_0}\circ...\circ g_{\alpha_k}(x)\to \pi(\alpha)$.
\end{itemize}

In this part we give a counterpart of this result (again, we follow the ideas from \cite{SS2}). We will always assume that a \gifs \ $\F=\{f_1,...,f_n\}$ satisfies $(S_1)$ and $A_\F$ is its fractal. All symbols have the same meaning as earlier. 



\begin{proposition}\label{zbieznosc do punktu}
For every $\alpha\in\mathbf{\Omega}$, the sequence $(A_{\alpha|_k})_{k\in\N^*}$ is a decreasing sequence of compact sets and $\on{diam}_d(A_{\alpha|_k}) \to 0$. In particular, there exists $x_\alpha\in X$ such that $\bigcap_{k\in\N^*} A_{\alpha|_k} = \{x_\alpha\}$.
\end{proposition}
\begin{proof}
By Theorem~\ref{kafelkowanie}(iia), the sequence $(A_{\alpha|_k})$ is decreasing sequence of compact sets. Hence it is enough to show that $\on{diam}_d(A_{\alpha|_k})\to 0$, but this follows from Theorem~\ref{kafelkowanie}(i) (as $L_{(s,1)}(\F)<1$).
\end{proof}

By the above result, we can define the mapping $\pi:\mathbf{\Omega} \to X$ by
$$\pi(\alpha) := x_\alpha,$$
where $x_\alpha$ is the unique point of $\bigcap_{k\in\N^*} A_{\alpha|_k}$.

The next result gives a counterpart of the mentioned theorem.
\begin{theorem}\label{abc17}
In the above frame (we underline that $\F$ satisfies $(S_1)$):
\begin{itemize}
\item[(i)] $\overline{\pi(\mathbf{\Omega})}=A_\F$.
\item[(ii)] If additionally $\F$ satisfies $(S_2)$, then $\pi(\mathbf{\Omega})=A_\F$.
\item[(iii)] $\pi:\mathbf{\Omega}\to X$ is continuous, provided on each $\Omega_k$ we consider the discrete topology (induced by $d_{k,s,1}$) and on $\mathbf{\Omega}$ the Tychonoff product topology.
\item[(iv)] If additionally $\F$ satisfies $(Q)$ (or, equivalently, $(P)$), then 
the map $\pi:\mathbf{\Omega}\to X$ is continuous provided we consider the metric $d_{(s,q)}$ (or, equivalently, $d_{(p,q)}$) on $\mathbf{\Omega}$.
\end{itemize}
\end{theorem}
\begin{remark}\emph{Denote by $\tau_1$ the topology on $\mathbf{\Omega}$ defined as in (iv), that is, the topology induced by any of metrics $d_{(p,q)}$ or (equivalently), $d_{s,q}$ ($q<1$). As was observed in Lemma \ref{filiplemma2}(ii), $(\mathbf{\Omega},\tau_1)$ is compact (it is a Tychonoff product of compact spaces $\Omega_k$).\\
Denote by $\tau_2$ the topology on $\mathbf{\Omega}$ defined as in (iii), that is, the Tychonoff product of discrete topologies on $\Omega_k$. Clearly, $\tau_1\subset \tau_2$ and $(\mathbf{\Omega},\tau_2)$ is not compact (in the case $n>1$).\\
Later we will give example that the thesis of (iv) does not hold under the assumption $(S_2)$.
Hence point~(iv) (whose proof will be the longest) says that $\pi$ has better properties if we additionally assume (Q) (or, equivalently, (P)), because it is continuous not only with respect to the topology $\tau_2$ but also with respect to the topology $\tau_1$}.
\end{remark}

Before we give the proof, we formulate an additional lemma. If $k\in\N$ and $\alpha=(\alpha_0,\alpha_1,...,\alpha_k),\beta=(\beta_0,\beta_1,...,\beta_k)\in{_k}\Omega$, then for $l=1,...,k$, define
$$C_l^{\alpha, \beta} := \left\{(i_0, ..., i_{l-1}): (\alpha_l)^{(i_0, ..., i_{l-1})} \neq (\beta_l)^{(i_0, ..., i_{l-1})}\right\}.$$ 
\begin{lemma}\label{alfabeta}
Let $k\in\N, x\in A_{k+1}$ and $\alpha,\beta\in{_k}\Omega$.
\begin{itemize}
\item[(i)] If $L_{(p,q)}(\F)\leq 1$, then
$$d\left(f_\alpha(x), f_{\beta}({x}) \right) \leq  \left\{ \begin{array}{ll} \on{diam}_d(A_{\mS}) \left( \sum_{i_0\in C_1^{\alpha, \beta}} q^{i_0} + ... + \sum_{(i_0, ..., i_{k-1}) \in C_k^{\alpha, \beta}} q^{i_0+...+i_{k-1}} \right )^{1/p}, \textrm{ if $\alpha_0 = \beta_0$},\\
\on{diam}_d(A_{\mS}), \textrm{ otherwise}. \end{array} \right. $$
\item[(ii)] If $L_{(s,q)}(\F)\leq 1$, then
$$d\left(f_\alpha(x), f_{\beta}({x}) \right) \leq  \left\{\begin{array}{ll}\on{diam}_d(A_{\mS})\max\{\sup\{q^{i_0}:i_0\in C_1^{\alpha,\beta}\},...,\sup\{q^{i_0+...+i_{k-1}}:(i_0,...,i_{k-1})\in C_k^{\alpha,\beta}\}\}, \textrm{ if $\alpha_0 = \beta_0$},\\
\on{diam}_d(A_{\mS}), \textrm{ otherwise}. \end{array} \right. $$
\end{itemize}
\end{lemma}

\begin{proof}
We will just prove the first assertion. The proof will be inductive.\\ First let $k=1$ and take any $x=(x_0,x_1,...)\in A_2$, and $\alpha=(\alpha_0,\alpha_1),\beta=(\beta_0,\beta_1) \in {_1}\Omega$.  
If $\alpha_0=\beta_0$ then
$$d\left(f_{\alpha}({x}), f_{\beta}({x}) \right) = d\left(f_{\alpha_0} \left(f_{\alpha_{1}^{(0)}}(x_0), f_{\alpha_{1}^{(1)}}(x_1), ... \right), f_{\beta_0} \left(f_{\beta_{1}^{(0)}}(x_0), f_{\beta_{1}^{(1)}}(x_1), ... \right) \right) \leq $$
$$
\leq d_{p,q} \left(f_{\alpha_{1}^{(0)}}(x_0), f_{\alpha_{1}^{(1)}}(x_1), ... \right), \left(f_{\beta_{1}^{(0)}}(x_0), f_{\beta_{1}^{(1)}}(x_1), ... \right) =
$$
$$= \left( \sum_{i=0}^\infty q^i d^p\left(f_{\alpha_{1}^{(i)}}(x_i), f_{\beta_{1}^{(i)}}(x_i) \right) \right)^{1/p} = $$
$$= \left( \sum_{i\notin C_1^{\alpha, \beta}} q^i d^p\left(f_{\alpha_{1}^{(i)}}(x_i), f_{\beta_{1}^{(i)}}(x_i) \right) + \sum_{i\in C_1^{\alpha, \beta}} q^i d^p\left(f_{\alpha_{1}^{(i)}}(x_i), f_{\beta_{1}^{(i)}}(x_i)\right) \right)^{1/p} = $$
$$=\left(\sum_{i\in C_1^{\alpha, \beta}} q^i d^p\left(f_{\alpha_{1}^{(i)}}(x_i), f_{\beta_{1}^{(i)}}(x_i) \right) \right)^{1/p} \leq \on{diam}_d(A_{\mS}) \left(\sum_{i\in C_1^{\alpha, \beta}} q^i \right)^{1/p},$$
where the last inequality comes from the fact that $f_j(A_1) \subset A_{\mS}$. Otherwise, if $\alpha_0 \neq \beta_0$, by Theorem~\ref{kafelkowanie} we simply have:
$$d\left(f_{\alpha}({x}), f_{\beta}({x}) \right) \leq \on{diam}_d(A_{\mS}).$$

Now assume that the thesis holds for $k \in \N$, and fix $x=(x_0,x_1,...)\in A_{k+2}$ and $\alpha=(\alpha_0,\alpha_1,...,\alpha_{k+1}), \beta=(\beta_0,\beta_1,...,\beta_{k+1})\in{_{k+1}}\Omega$. Additionally define $$C_*^{\alpha, \beta} := \{i: \alpha(i)_0 \neq \beta(i)_0\},$$
where $\alpha(i)_0,...,\alpha(i)_k$ are coefficients of $\alpha(i)$, i.e., $\alpha(i)=(\alpha(i)_0,...,\alpha(i)_k)$, and similarly for $\beta(i)$.\\
If $\alpha_0 = \beta_0$, then we have:
$$d\left(f_{\alpha}({x}), f_{\beta}({x}) \right) = $$
$$= d\left(f_{\alpha_0} \left(f_{\alpha(0)}(x_0), f_{\alpha(1)}(x_1), ... \right), f_{\beta_0} \left(f_{\beta(0)}(x_0), f_{\beta(1)}(x_1), ... \right) \right) \leq $$
$$\leq \left( \sum_{i\in\N^*} q^i d^p\left(f_{\alpha(i)}(x_i), f_{\beta(i)}(x_i) \right) \right)^{1/p} = $$
$$=  \left( \sum_{i\in C_*^{\alpha, \beta}} q^i d^p\left(f_{\alpha(i)}(x_i), f_{\beta(i)}(x_i) \right) +  \sum_{i\notin C_*^{\alpha, \beta}} q^i d^p\left(f_{\alpha(i)}(x_i), f_{\beta(i)}(x_i) \right) \right)^{1/p}$$
Fix $i\in\N^*$. Observe that for every $l=0,...,k$, we have $\alpha(i)_l = \alpha_{l+1}^{(i)}$ and 
$$(\alpha(i)_l)^{(i_0, ..., i_{l-1})} = (\alpha_{l+1}^{(i)})^{(i_0, ..., i_{l-1})} = \alpha_{l+1}^{(i, i_0, ..., i_{l-1})},$$
and similarly for $\beta$.\\ 
Hence $C_*^{\alpha, \beta} = C_1^{\alpha, \beta}$ and for every $l=1,...,k$, 
$$C^{\alpha(i), \beta(i)}_l = \left\{(i_0, ..., i_{l-1}): \alpha_{l+1}^{(i, i_0, ... i_{l-1})} \neq \beta_{l+1}^{(i, i_0, ... i_{l-1})}\right\}=\{(i_0,..., i_{l-1}):(i,i_0,...,i_{l-1})\in C^{\alpha,\beta}_{l+1}\}.$$ 
Observe also that $(\alpha_0,...,\alpha_{l+1})(i)=(\alpha(i)_0,...,\alpha(i)_l)$. 
Thus by the above observations and the inductive assumption, we can continue the above computations:
$$\leq  \left( \sum_{i\in C_1^{\alpha, \beta}} q^i \on{diam}_d^p(A_{\mS}) + \right.  \;\;\;\;\;\;\;\;\;\;\;\;\;\;\;\;\;\;\;\;\;\;\;\;\;\;\;\;\;\;\;\;\;\;\;\;\;\;\;\;\;\;\;\;\;\;\;\;\;\;\;\;\;\;\;\;\;\;\;\;\;\;\;\;\;\;\;\;\;\;\;\;$$
$$\;\;\;\;\;\;\;\;\;\;\;\;\;\;\;\;\;\;\;
\left. + \sum_{i\notin C_1^{\alpha, \beta}} q^i  \on{diam}_d^p(A_{\mS}) \left( \sum_{i_0\in C_1^{\alpha(i), \beta(i)}} q^{i_0} + ... + \sum_{(i_0, ..., i_{k-1}) \in C_k^{\alpha(i), \beta(i)}} q^{i_0+...+i_{k-1}} \right ) \right)^{1/p} \leq $$
$$\leq  \left( \sum_{i\in C_1^{\alpha, \beta}} q^i \on{diam}_d^p(A_{\mS}) +  \on{diam}_d^p(A_{\mS}) \sum_{i\in\N^*} q^i  \left( \sum_{i_0\in C_1^{\alpha(i), \beta(i)}} q^{i_0} + ... + \sum_{(i_0, ..., i_{k-1}) \in C_k^{\alpha(i), \beta(i)}} q^{i_0+...+i_{k-1}} \right ) \right)^{1/p} = $$
$$= \left( \sum_{i\in C_1^{\alpha, \beta}} q^i \on{diam}_d^p(A_{\mS}) +  \on{diam}_d^p(A_{\mS})\left(\sum_{(i,i_0) \in C_2^{\alpha, \beta}} q^{i+i_0} + ... + \sum_{(i, i_0, ..., i_{k-1}) \in C_{k+1}^{\alpha, \beta}} q^{i+i_0+...+i_{k-1}} \right) \right)^{1/p} = $$
$$= \on{diam}_d(A_{\mS}) \left( \sum_{i\in C_1^{\alpha, \beta}} q^i  +  \sum_{(i_0,i_1) \in C_2^{\alpha, \beta}} q^{i_0+i_1} + ... + \sum_{(i_0, i_1, ..., i_{k}) \in C_{k+1}^{\alpha, \beta}} q^{i_0+...+i_k} \right)^{1/p}. $$
Clearly, if $\alpha_0 \neq \beta_0$ we have 
$$d\left(f_\alpha(x), f_\beta(x) \right) \leq \on{diam}_d(A_{\mS}),$$
which ends the proof.
\end{proof}
We are ready to give the proof of Theorem \ref{abc17}.
\begin{proof}
Ad(i). Take any $\alpha=(\alpha_0,\alpha_1,...)\in\mathbf{\Omega}$. Then $\pi(\alpha)\in \bigcap_{k\in\N^*} A_{\alpha|_k} \subset A_{(\alpha_0)} \subset A_{\mS}$. We proved that $\pi(\mathbf{\Omega}) \subset A_{\mS}$ and therefore $\overline{\pi(\mathbf{\Omega})} \subset A_\mS$.\\
Now let $x\in A_{\mS}$ and $\varepsilon>0$. We will construct a sequence $\alpha\in \mathbf{\Omega}$ such that $d(x, \pi(\alpha))<\varepsilon$. We will proceed inductively. 
Since $A_{\mS} = \bigcup_{i=1}^n A_{(i)}$, there exists some $\alpha_0 \in \{1,...,n\} = \Omega_0$ such that $x\in A_{(\alpha_0)}$. Let $y_0:=x$. By Theorem~\ref{kafelkowanie}(iia): $A_{(\alpha_0)} = \overline{\bigcup_{\alpha_1 \in \Omega_1} A_{(\alpha_0, \alpha_1)}}$, hence there exist $\alpha_1 \in \Omega_1$ and $y_1 \in A_{(\alpha_0, \alpha_1)}$ such that $d(x, y_1) = d(y_0, y_1) < \frac{\varepsilon}{4}$. Assume that for some $k\in\N$ we defined $\alpha_k\in\Omega_k$ and $y_k \in A_{(\alpha_0, ..., \alpha_k)}$ such that $d(y_{k-1}, y_k) < \frac{\varepsilon}{2^{k+1}}$. By Theorem~\ref{kafelkowanie}(iia): $A_{(\alpha_0, ..., \alpha_k)} = \overline{\bigcup_{\alpha_{k+1} \in \Omega_{k+1}} A_{(\alpha_0, ..., \alpha_k, \alpha_{k+1})}}$, hence there exists $\alpha_{k+1} \in \Omega_{k+1}$ and $y_{k+1} \in A_{(\alpha_0, ..., \alpha_k, \alpha_{k+1})}$ such that $d(y_k, y_{k+1}) < \frac{\varepsilon}{2^{k+2}}$. We defined sequences $(\alpha_k)$ and $(y_k)$. Set $\alpha := (\alpha_0, \alpha_1, ...) \in \mathbf{\Omega}$. As $y_k\in A_{\alpha|_{k+1}}$ for all $k$, $\bigcap_{k\in\N^*}A_{\alpha|_k}=\{\pi(\alpha)\}$ and $\on{diam}_d(A_{\alpha|_k})\to0$, we have $y_k\to \pi(\alpha)$. In particular, there exists $k_0\in\N$ such that $d(y_{k_0},\pi(\alpha))<\frac{\varepsilon}{2}$. Thus
$$
d(x,\pi(\alpha))\leq d(x,y_1)+d(y_1,y_2)+...+d(y_{k_0-1},y_{k_0})+d(y_{k_0},\pi(\alpha))<\frac{\varepsilon}{2}+\frac{\varepsilon}{2}=\varepsilon
$$

Since $\varepsilon$ was taken arbitrarily we get $x \in \overline{\pi(\mathbf{\Omega})}$. 
\newline
Ad(ii). Observe that by Theorem~\ref{kafelkowanie}(iiib) in each step of the induction in the proof of (i) we can choose $\alpha_{k+1} \in \Omega_{k+1}$ such that $x \in A_{(\alpha_0, ..., \alpha_k, \alpha_{k+1})}$. Then $\bigcap_{k\in\N^*} A_{\alpha|_k} = \{x\}$ and therefore $\pi(\mathbf{\Omega}) = A_\mS$. 
\newline
Ad(iii). Let $\alpha=(\alpha_0,\alpha_1,...)\in\mathbf{\Omega}$ and $U\ni \pi(\alpha)$ be any open set. Since $\bigcap_{k\in\N^*}A_{\alpha|_k}=\{\pi(\alpha)\}$, there is $k_0$ such that $A_{\alpha|_{k_0+1}}\subset U$. Now let $$V=\{\alpha_0\}\times...\times\{\alpha_{k_0}\}\times\Omega_{k_0+1}\times\Omega_{k_0+2}\times... .$$
Then $V$ is an open set containing $\alpha$, and for every $\beta\in V$, we have $\pi(\beta)\in A_{\beta|_{k_0+1}}=A_{\alpha|_{k_0+1}}\subset U$. This gives the desired continuity.
\newline
Ad(iv). Assume that (P) holds.
Let $C_l^{\alpha, \beta}$ be defined as earlier. 
Observe that, if $\alpha, \beta \in \mathbf{\Omega}$ are such that $\alpha_0=\beta_0$, then by Lemma \ref{filiplemma1} we have that for any $k \in \N$,
$$d_{(p,q)}(\alpha, \beta) = \left( \sum_{i\in\N^*} \left(\frac{1-q}{2}\right)^i{d_{i,p,q}^p(\alpha_i, \beta_i)} \right)^{1/p} \geq$$ $$\geq \left(\left(\frac{1-q}{2}\right)^1{\sum_{i_0\in C_1^{\alpha, \beta}} q^{i_0}} + ... + \left(\frac{1-q}{2}\right)^k{\sum_{(i_0, ..., i_{k-1})\in C_k^{\alpha, \beta}} q^{i_0+...+i_{k-1}}} \right)^{1/p}.$$
Let $\varepsilon>0$. 

By Theorem~\ref{kafelkowanie}(i) there exists $k\in \N$ such that for any $\gamma \in {_k}\Omega$, we have that $\on{diam}_d(A_\gamma)<\frac{\varepsilon}{3}$. 
Take $\delta := \min\left\{1, \frac{\varepsilon}{3}\cdot\frac{\left(\frac{1-q}{2}\right)^{\frac{k}{p}}}{\on{diam}_d(A_{\mS})}\right\}$ (here we must assume $\on{diam}_d(A_\F)>0$; if $\on{diam}_d(A_\F)=0$, then (iv) clearly holds). Let $\alpha, \beta \in \mathbf{\Omega}$ be such that $d_{(p,q)}(\alpha, \beta)<\delta$. Since $\delta<1$, we have $\alpha_0=\beta_0$. Now let $x \in A_{k+1}$. By Lemma~\ref{alfabeta}:
$$d\left(f_{\alpha|_{k}}({x}), f_{\beta|_{k}}({x})\right) \leq \on{diam}_d(A_{\mS}) \left( \sum_{i_0\in C_1^{\alpha, \beta}} q^{i_0} + ... + \sum_{(i_0, ..., i_{k-1}) \in C_k^{\alpha, \beta}} q^{i_0 + ... + i_{k-1}} \right )^{1/p} \leq $$
$$\leq \on{diam}_d(A_{\mS})\left(\frac{1-q}{2}\right)^{-\frac{k}{p}} \left(\left(\frac{1-q}{2}\right)^1{\sum_{i_0\in C_1^{\alpha, \beta}} q^{i_0}} + ... + \left(\frac{1-q}{2}\right)^k{\sum_{(i_0, ..., i_{k-1})\in C_k^{\alpha, \beta}} q^{i_0+...+i_{k-1}}} \right)^{1/p} \leq $$
$$\leq \on{diam}_d(A_{\mS})\left(\frac{1-q}{2}\right)^{-\frac{k}{p}}d_{(p,q)}(\alpha, \beta) \leq \frac{\varepsilon}{3}.$$
On the other hand, $\pi(\alpha),f_{\alpha|_{k}}({x})\in A_{\alpha|_{k}}$ and $\pi(\beta),f_{\beta|_{k}}({x})\in A_{\beta|_{k}}$, so we have
$$d(\pi(\alpha), \pi(\beta)) \leq d\left(\pi(\alpha), f_{\alpha|_{k}}({x})\right) + d\left(f_{\alpha|_{k}}({x}), f_{\beta|_{k}}({x})\right) + d\left(f_{\beta|_{k}}({x}), \pi(\beta) \right)\leq $$ $$ \leq \on{diam}_d(A_{\alpha|_{k}})+\frac{\varepsilon}{3}+\on{diam}_d(A_{\beta|_{k}}) <\varepsilon.$$
This gives the continuity of $\pi$.
\end{proof} 


Now the promised examples:
\begin{example}\emph{
Let $\F=\{f_1,f_2\}$ be a \gifs \ from Example \ref{exampleee1}(1). It satisfies $(S_2)$. On the code space~$\mathbf{\Omega}$ (where $\Omega=\{1,2\}$) consider the topology induced by any of metrics $d_{(p,q)}$ or $d_{(s,q)}$, $q<1$.\\
For $k\in\N^*$, let $\alpha_k\in\Omega_k$ be such that all coefficients $\alpha_k^{(i_0,...,i_{k-1})}=1$. Then for every $n\in\N$, let
$$
\alpha_{(n)}:=(1,1,...,1,2,1,...),
$$
where $2$ stands on the $n$'th place. Clearly, the sequence $\left((\alpha_0,\alpha_{(n)},\alpha_2,\alpha_3,...)\right)_{n\in\N}$ converges (with respect to $d_{(p,q)}$) to the sequence $(\alpha_0,\alpha_1,\alpha_2,...)$.\\
For $0\leq b<a\leq 1$, set $[b,a]_X:=[b,a]\cap X$, where $X=\left\{\frac{1}{k}:k\in\N\right\}\cup\{0\}$ is the underlying space here.
Now observe that for every nonempty $[b,a]_X$, 
$$f_1([b,a]_X\times[b,a]_X\times....)\subset\left[\frac{1}{2}b,\frac{1}{2}a\right]_X \textrm{\; \; \; and \;\;\;} f_2([b,a]_X\times[b,a]_X\times....) = \left\{1\right\}.$$
From this we can easily see that $A_{(\alpha_0,...,\alpha_k)} \subset \left[0,\frac{1}{2^{k+1}}\right]_X$ and $\pi(\alpha_0,\alpha_1,...)=0$. On the other hand, for every $n\in\N^*$,
$$
\pi(\alpha_0,\alpha_{(n)},\alpha_2,...)\in A_{\left(\alpha_0,\alpha_{(n)}\right)} \subset$$ $$
\subset f_1\left(f_1\left(X\times X\times...\right)\times...\times f_1\left(X\times X\times...\right)\times f_2\left(X\times X\times...\right)\times f_1\left(X\times X\times...\right)\times...\right) \subset$$ $$\subset {
f_1\left(X\times X\times ...\times X\times \left\{1\right\} \times X\times...\right)}= \left\{\frac{1}{2}\right\}.
$$
In particular, $\pi(\alpha_0,\alpha_{(n)},\alpha_2,...) = \frac{1}{2}$, hence $\pi(\alpha_0,\alpha_{(n)},\alpha_2,...)\not\to \pi(\alpha_0,\alpha_1,...)$ so $\pi$ is not continuous with respect to considered topology.\\
In fact, it can be observed that if $(\alpha_0,\alpha_1,...\alpha_k)\in{_k}\Omega$ is such that $\alpha_0=1$ and for $1\leq j<k$, all $\alpha_j^{(i_0,...,i_{j-1})}=1$, but $\alpha_k^{(i_0,...,i_{k-1})}=2$ for some $i_0,...,i_{k-1}$, then $A_{(\alpha_0,...,\alpha_k)}=\left\{\frac{1}{2^k}\right\}$. Moreover, $A_{(2)}=\{1\}$.
}\end{example}
The next example bases on the one given by Secelean in \cite[Example 3.2]{Se}
\begin{example}\emph{
Consider $X := \R$ endowed with the Euclidean metric. Then ($\ell_\infty(\R), d_{s,1}$) is a Banach space which is known to be non-reflexive. Hence by the James theorem there exists a continuous linear functional $f: \ell_\infty(\R) \to \R$ such that $\left\|f \right\| = 1$ and its norm is not attained. In particular, $f(\prod_{k\in\N^*}[-1,1]) = (-1,1)$. Define $g_1, g_2$ as:
$$g_1 := \frac{1}{2} f - \frac{1}{2} \ \ \ \ \textrm{    and    } \ \ \ \ g_2 := \frac{1}{2} f + \frac{1}{2}.$$
Clearly $\left\|g_1 \right\| = \left\|g_2 \right\| = \frac{1}{2}$ and (C1) is fulfilled (since an image of any set in $\ell_\infty(\K(\R))$ by $g_i$ is bounded and therefore its closure is compact in $\R$). Hence \gifs\ $\mathcal{G} := \{g_1, g_2\}$ fulfills $(S_1)$ and by Theorem~\ref{se2} it has a unique attractor $A_\mathcal{G}$. Since $g_1(\prod_{k\in\N^*}[-1,1]) = (-1, 0)$ and $g_2(\prod_{k\in\N^*}[-1,1]) = (0, 1)$ we get that $A_\mathcal{G} = [-1,1]$. On the other hand we will show that $\pi(\mathbf{\Omega}) \subset (-1, 0) \cup (0,1)$.\\ 
Take any $\beta=(\beta_0,\beta_1,...)\in\Omega_1$. By the above observations, for every $i\in\N^*$, either $g_{\beta_i}(\prod_{k\in\N^*}[-1,1])=(-1,0)$ or $g_{\beta_i}(\prod_{k\in\N^*}[-1,1])=(0,1)$. Hence there exist $\varepsilon_0,\varepsilon_1,...\in\{-1,1\}$ such that
$$
f\left(g_{\beta_0}\left(\prod_{k\in\N^*}[-1,1]\right)\times g_{\beta_1}\left(\prod_{k\in\N^*}[-1,1]\right)\times...\right)=f\left(\left(-\frac{1}{2},\frac{1}{2}\right)\times\left(-\frac{1}{2},\frac{1}{2}\right)\times...\right)+f\left(\frac{\varepsilon_0}{2},\frac{\varepsilon_1}{2},...\right)=
$$
$$
=\frac{1}{2}f\big( \left(-1,1\right)\times\left(-1,1\right)\times... \big)+\frac{1}{2}f\left({\varepsilon_0},{\varepsilon_1},...\right)\subset\left[\frac{1}{2}(c-1),\frac{1}{2}(c+1)\right]
$$
where $c:=f\left({\varepsilon_0},{\varepsilon_1},...\right)\in(-1,1)$. Hence
$$
g_1\left(g_{\beta_0}\left(\prod_{k\in\N^*}[-1,1]\right)\times g_{\beta_1}\left(\prod_{k\in\N^*}[-1,1]\right)\times...\right)\subset \frac{1}{2}\left[\frac{1}{2}(c-1),\frac{1}{2}(c+1)\right]-\frac{1}{2}=\left[\frac{1}{4}c-\frac{3}{4},\frac{1}{4}c-\frac{1}{4}\right]
$$
In particular, $A_{(1,\beta)}\subset \left[\frac{1}{4}c-\frac{3}{4},\frac{1}{4}c-\frac{1}{4}\right]\subset (-1,0)$. 
 Similarly we show that $A_{(2,\beta)}\subset (0,1)$. Since $\beta$ was taken arbitrarily, we have that $\pi(\Omega)\subset(-1,0)\cup(0,1)$. 
}\end{example}

Finally, we show that appropriate sequence of sets converges to $\pi(\alpha)$.

\begin{proposition}\label{zbieznosc zbiorow}
For every $\alpha\in\mathbf{\Omega}$ and every bounded set $D\subset X$, the sequence $(f_{\alpha|_{k}}(D_{k+1}))$ (where $(D_k)$ is defined as in (\ref{abc0})) converges to $\{\pi(\alpha)\}$ with respect to the Hausdorff-Pompeiu metric.
\end{proposition}

\begin{proof}
Define 
$$D':=D \cup A_{\mS}$$
and let $(D_k')$ be defined by (\ref{abc0}). Clearly, 
for any $k\in\N^*$, $A_k\subset D_k'$, so
$$\pi(\alpha)\in A_{\alpha|_k} = \overline{f_{\alpha|_k}(A_{k+1})} \subset \overline{f_{\alpha|_k}(D'_{k+1})}.$$
On the other hand, by Lemma~\ref{mniejsze srednice}, $\on{diam}_d(\overline{f_{\alpha|_k}(D'_{k+1})}) \to 0$ (as $L_{(s_1)}(\F) <1$), so $\bigcap_{k\in\N^*} \overline{f_{\alpha|_{k}} (D_{k+1}')} = \{\pi(\alpha)\}$.
In particular, for every $r>0$ there is $k_0\in\N^*$ such that for $k\geq k_0$, $\on{diam}_d(\overline{f_{\alpha|_k}(D'_{k+1})}) < r$. Hence (by the fact that $D_{k+1}\subset D'_{k+1}$), we have $\overline{f_{\alpha|_k}(D_{k+1})} \subset \overline{ f_{\alpha|_k} (D'_{k+1})} \subset B(\pi(\alpha), r)$, where $B(\pi(\alpha), r)$ states for the closed ball.
\end{proof}

If ${x}=(x_0, x_1, ...) \in \ell_\infty(X)$, then we define the sequence $(x_k)$ in the following inductive way:
\begin{equation}\label{abc18}
\mathbf{x}_1:={x},
\end{equation}
\begin{equation}\label{abc19}
\mathbf{x}_{k+1}:=(\mathbf{x}_{k},\mathbf{x}_{k},\mathbf{x}_{k},...) \textrm{ for }k\geq 1.
\end{equation}
Clearly, $\mathbf{x}_k\in X^\infty_k$ for every $k\in\N$.
\begin{corollary}\label{t12}
For every ${x}\in \Xinf$ and $\alpha\in\mathbf{\Omega}$, we have that $\lim_{k\rightarrow\infty}f_{\alpha|_{k}}(\mathbf{x}_{k+1})=\pi(\alpha)$.
\end{corollary}
\begin{proof}
It simply follows from Proposition~\ref{zbieznosc zbiorow}. Take $D:=\{x_0, x_1, ...\}$. Then $D$ is bounded subset of $X$, and for each $k \in \N$, $\mathbf{x}_k\in D_{k}$, where $D_{k}$ is as in Proposition~\ref{zbieznosc zbiorow}. In particular, $f_{\alpha|_{k}}(\mathbf{x}_{k+1})\in f_{\alpha|_{k}}(D_{k+1})$ and the result follows.
\end{proof}

\subsection{The relationships between a \gifs$\;\F$ and the canonical \gifs$\;\F_\mathbf{\Omega}$}
If $\{g_1,...,g_n\}$ is an IFS consisting of Banach contractions, then for every $k\in\N^*$ and every $\alpha_0,...,\alpha_k\in\{1,...,n\}$, we have
\begin{equation*}
g_{\alpha_0}\circ...\circ g_{\alpha_k}\circ\pi=\pi\circ\sigma_{\alpha_0}\circ...\circ\sigma_{\alpha_k},
\end{equation*}
where $\pi:\prod_{i=0}^\infty\{1,...,n\}\to X$ is the canonical map and $\sigma_1,...,\sigma_n$ are shifts (recalled already in Section~\ref{subsection52}).\\
In this section we show that the same relations between a \gifs \ and the canonical \gifs \ on the code space (see \cite[Theorem 3.11]{SS2} for the GIFS's case).\\
Set $\Pi:=\mathbf{\Omega}$ (for $\Omega=\{1,...,n\}$), and let $\Pi_k$, $k\in\N^*$ be defined by (\ref{abc0}).\\
For every $k\in\N^*$, we will define a family of mappings $\mathcal{T}_k=\{\tau_\alpha: \Pi_{k+1} \to \mathbf{\Omega}: \alpha\in {_k\Omega}\}$ by induction with respect to $k$. For $k=0$ we have already defined this family - this is just $\mathcal{T}_0=\{\tau_1, ..., \tau_n\}$ (the canonical \gifs\ on $\mathbf{\Omega}$ from Section 5.2). For $\alpha = (\alpha_0, \alpha_1, ..., \alpha_k, \alpha_{k+1})\in {_{k+1}}\Omega$, we set $\tau_\alpha: \Pi_{k+2} \to \mathbf{\Omega}$ by 
$$\tau_\alpha(\beta_0, \beta_1, ... ) := \tau_{\alpha_0}\left(\tau_{\alpha(0)}(\beta_0), \tau_{\alpha(1)}(\beta_1), ... \right).$$
Observe that, in fact, the families $\mathcal{T}_k$ are defined as $\F_k$, for the \gifs$\;\{\tau_1,...,\tau_n\}$ (see Section 6.1).\\
Now let $\F=\{f_1,...,f_n\}$ be a \gifs\ which satisfies $(S_1)$, $A_\F$ be its attractor and $(A_k)$ be the sequence defined as earlier (for $A:=A_\F$). Let $\{\pi_{k}:\Pi_{k+1} \rightarrow X_{k+1}:k\in\N^*\}$ be the family of mappings defined by the following inductive formula:
$$\pi_0(\alpha_0, \alpha_1, ...):=(\pi(\alpha_0), \pi(\alpha_1), ...),$$
$$\pi_{k+1}(\alpha_0, \alpha_1, ...):=(\pi_{k}(\alpha_{0}), \pi_{k}(\alpha_{1}), ...) \textrm{ for $k\geq 0$}.$$
\begin{theorem}\label{diag2}
In the above frame, for every $k\in \N^*$ and $\alpha \in {_k}\Omega$, $f_{\alpha}\circ \pi_{k} = \pi \circ \tau_{\alpha}$.
 \end{theorem} 
 \begin{proof}
 We will proceed inductively with respect to $k$. Let $k=0$ and $\alpha=(\alpha_0, \alpha_1, ...)\in \Pi_1 =\prod_{i\in\N^*}\mathbf{\Omega}$. Let $(\mathbf{x}_{k})$ be a sequence built as in (\ref{abc18}) and (\ref{abc19}) for some arbitrarily taken $x\in A_1$. 
Fix $\alpha=(\alpha_0,\alpha_1,...)\in\Pi_1={\prod_{i\in\N^*}\mathbf{\Omega}}$. For every $j,l\in\N^*$, we have
$$f_{{\alpha_j}|_l}(\mathbf{x}_{l+1})\in f_{{\alpha_j}|_l}(A_{l+1})=A_{{\alpha_j}|_l}$$
and by definition, $\pi(\alpha_j)\in A_{{\alpha_j}|_l}$. Hence by Theorem~\ref{kafelkowanie}(i),
$$
d(f_{{\alpha_j}|_l}(\mathbf{x}_{l+1}),\pi(\alpha_j))\leq \on{diam}_d(A_{{\alpha_j}|_l})\leq L_{(s,1)}(\F)^{l+1}\on{diam}_d(A_\F).
$$ 
In particular, for every $i=1,...,n$,
$$
d(f_i(\pi(\alpha_0),\pi(\alpha_1),...),f_i(f_{{\alpha_0}|_l}(\textbf{x}_{l+1}),f_{{\alpha_1}|_l}(\textbf{x}_{l+1}),...))\leq
$$
$$\leq \sup\left\{d\left(\pi(\alpha_j),f_{{\alpha_j}|_l}(\textbf{x}_{l+1})\right):j\in\N^*\right\}\leq L_{(s,1)}(\F)^{l+1}\on{diam}_d(A_\F).
$$
This means that
$$
\lim_{l \to \infty} \left(f_{i}\left(f_{\alpha_{0}|_{l}}(\mathbf{x}_{l+1}), f_{\alpha_{1}|_{l}}(\mathbf{x}_{l+1}), ... \right)\right)=f_i \left(\pi(\alpha_0), \pi(\alpha_1), ...\right) .
$$
Hence by the above and Corollary~\ref{t12} we have
$$f_i\circ \pi_0(\alpha)=f_i \left(\pi(\alpha_0), \pi(\alpha_1), ...\right) =$$ $$= \lim_{l \to \infty} \left(f_{i}\left(f_{\alpha_{0}|_{l}}(\mathbf{x}_{l+1}), f_{\alpha_{1}|_{l}}(\mathbf{x}_{l+1}), ... \right)\right)= \lim_{l \to \infty} (f_{\tau_{i}(\alpha)\vert_{l+1}}(\mathbf{x}_{l+2}))=\pi(\tau_{i}(\alpha))=\pi\circ \tau_i(\alpha).$$
Thus we get the thesis for $k=0$.
Now assume that for some $k \in \N^*$ we have the thesis and take any $\beta=(\beta_0,...,\beta_{k+1}) \in {}_{k+1}\Omega$ and some $\alpha=(\alpha_{0}, \alpha_1, ...)\in \Pi_{k+2}$. Then:
$$ f_{\beta}\circ \pi_{k+1}(\alpha) = f_{\beta}\big(\pi_{k}(\alpha_0), \pi_{k}(\alpha_1), ... \big) = f_{\beta_{0}}\left( f_{\beta(0)}(\pi_{k}(\alpha_{0})), f_{\beta(1)}(\pi_{k}(\alpha_{1})), ...\right) =  $$ 
$$ = f_{\beta_{0}} \left( (f_{\beta(0)}\circ \pi_{k})(\alpha_{0}), (f_{\beta(1)}\circ \pi_{k}) (\alpha_{1}), ... \right) = $$
$$ = f_{\beta_{0}} \left( (\pi\circ \tau_{\beta(0)}) (\alpha_{0}), (\pi \circ \tau_{\beta(1)})(\alpha_{1}), ... \right) = $$ 
$$ = f_{\beta_{0}} \big( \pi(\tau_{\beta(0)} (\alpha_{0})), \pi(\tau_{\beta(1)}(\alpha_{1})), ... \big) = \left( f_{\beta_{0}}\circ \pi_{0} \right) \left( \tau_{\beta(0)} (\alpha_{0}), \tau_{\beta(1)} (\alpha_{1}), ... \right) = $$
$$ = \left( \pi \circ \tau_{\beta_{0}} \right) \left( \tau_{\beta(0)} (\alpha_{0}), \tau_{\beta(1)} (\alpha_{1}), ... \right) = \pi \big( \tau_{\beta_{0}}\left(\tau_{\beta(0)} (\alpha_{0}), \tau_{\beta(1)} (\alpha_1), ... \right) \big) =$$
$$ = \pi(\tau_{\beta}(\alpha))=\pi  \circ \tau_{\beta}(\alpha). $$
\end{proof}

\section{An example}
In this section we will construct a Cantor set on the plane which is an attractor of some GIFS$_\infty$, but cannot be generated by any GIFS. The construction will follow the ideas from \cite{S} and \cite{CR}.\\
At first, fix numbers $K,q\in(0,1)$, and a nondecreasing sequence $(m_k)$ of positive integers. Then define decreasing sequences $(p_k)$ and $(a_k)$ of positive reals such that 
 \begin{equation}\label{fin1}2p_0+a_0=1\end{equation} and for every $k=0,1,...$,
\begin{equation}\label{fin2}
\sqrt{4^{m_0\cdots m_k}}p_{k+1}+(\sqrt{4^{m_0\cdots m_k}}-1)a_{k+1}=p_{k}
\end{equation}
and
\begin{equation}\label{fin3}
\frac{p_k}{a_k}<\frac{Kq^{m_k}}{\sqrt{2}}
\end{equation}
Clearly, such a choice is possible - it is enough to take
$$
a_0:=\frac{2\sqrt{2}}{2Kq^{m_0}+2\sqrt{2}},\;\;\;\;p_0:=\frac{Kq^{m_0}}{2Kq^{m_0}+2\sqrt{2}}
$$ 
and for $k\geq 0$,
$$
a_{k+1}:=\frac{2\sqrt{2}}{\sqrt{4^{m_0\cdots m_k}}Kq^{m_{k+1}}+2\sqrt{2}(\sqrt{4^{m_0\cdots m_k}}-1)}p_k,\;\;p_{k+1}:=\frac{Kq^{m_{k+1}}}{\sqrt{4^{m_0\cdots m_k}}Kq^{m_{k+1}}+2\sqrt{2}(\sqrt{4^{m_0\cdots m_k}}-1)}p_k
$$ 
Now define $\tilde{\Omega}_k$, $k=0,1,2,...$ by the following inductive formula:
$$
\tOm_0:=\{1,...,4\},
$$
$$
\tOm_{k+1}:=\tOm_k^{m_k}\;\;\;\;\;\;\;\;\mbox{(the Cartesian product of $m_k$ copies of $\tilde{\Omega}_k$)}.
$$
Observe that the cardinalities $|\tOm_0|=4$ and $|\tOm_k|=4^{m_0\cdots m_{k-1}}$ for $k\geq 1$.\\
Finally, set 
$${_k}\tOm:=\tOm_0\times...\times\tOm_k,\;\;\mbox{and}\;\;\tOm:=\tOm_0\times\tOm_1\times\tOm_2\times...$$
We see that the above construction resembles the construction of code space for \gifs, but the difference is that at each step we take a finite product. In fact, if all $m_k$ are equal (to some value $m$), then we get the code space for GIFSs of order $m$ consisting of four mappings (\cite[Section 2]{SS2}). As will be seen, the interesting case is when $m_k\to\infty$.\\
Now we choose a family $\{I_\alpha:\alpha\in{_k}\tOm,\; k\in\N^*\}$ of squares on the plane such that:
\begin{equation*}
I_{(1)}=[0,p_0]\times[0,p_0],\;\;\;I_{(2)}=[p_0+a_0,1]\times[0,p_0],
\end{equation*}
\begin{equation*}
I_{(3)}=[0,p_0]\times[p_0+a_0,1],\;\;\;I_{(4)}=[p_0+a_0,1]\times[p_0+a_0,1]
\end{equation*}
and for every $k\in\N^*$ and $\alpha\in{_k}\tOm$,  the following conditions hold:
\begin{itemize}
\item[(i)] $|I_\alpha|=p_{k}$, where $|I|$ denotes the length of a side of a square $I$; 
\item[(ii)] the squares $I_{\alpha\hat{\;}\beta}$, for $\beta\in \tOm_{k+1}$ are pairwise disjoint subsets of $I_\alpha$, uniformly distributed on $I_\alpha$ in the sense that if $I_\alpha=[a,a+p_k]\times[c,c+p_k]$, then each $I_{\alpha\hat{\;}\beta}$ is of the form (denote $h_{k+1} := p_{k+1} + a_{k+1}$):
$$[a+i h_{k+1},a+i h_{k+1}+p_{k+1}]\times [c+j h_{k+1},c+j h_{k+1} + p_{k+1}]$$
\ \ \ \ for some $i,j=0,...,\sqrt{4^{m_0\cdots m_k}}-1$.
\end{itemize}
Note that the construction can be handled by (\ref{fin1}) and (\ref{fin2}).\\
Obviously, for every $\alpha\in\tOm$, the set $\bigcap_{k\in\N^*}I_{\alpha\vert_k}$ is a singleton.  Denoting its unique element by $x_\alpha$, define 
$$
C:=\{x_\alpha:\alpha\in\tOm\}.
$$
Since $C=\bigcap_{k\in\N^*}\bigcup_{\alpha\in{_k}\tOm}I_\alpha$, it is closed in $X$. In fact, $C$ is a Cantor-type set.
\begin{remark}\emph{Note that the set $C$ is defined exactly as in \cite{S}, with $n_1=4$ and $n_k=4^{m_0\cdots m_{k-2}}$ for $k\geq 2$. There is just a slight difference in notation - the sets $\tOm_k$ are replaced with $\{1,...,n_{k+1}\}$.}
\end{remark}
Consider the Euclidean metric on $C$. The above remark allows us to use \cite[Lemma 5]{S} and state the following:
\begin{fact}\label{fak1}
If $m\in\N$ and $g:C^m\to C$ is a generalized weak contraction (when considering the maximum metric on $C^m$), then for every $\alpha_0,...,\alpha_{m-1}\in{_k}\Omega$, there is $\beta\in{_{k+1}}\Omega$ such that
$$
g((I_{\alpha_0}\cap C)\times...\times (I_{\alpha_{m-1}}\cap C))\subset I_\beta\cap C.
$$
\end{fact}
By \emph{a generalized weak contraction} we understand a map $f:X^m\to X$ satisfying for $(x_0,...,x_{m-1})\neq(y_0,...,y_{m-1})$:
$$
d(f(x_0,...,x_{m-1}),(y_0,...,y_{m-1}))<\max\{d(x_0,y_0),...,d(x_{m-1},y_{m-1})\}.
$$
In fact, \cite[Lemma 5]{S} is formulated for generalized Matkowski contractions but, as is noted before \cite[Remark 1]{S}, these notions coincide for compact spaces.\\
Now for every $i=1,2,3,4$, define the mapping $i:\prod_{k=0}^\infty\tOm\to\tOm$ by 
$$
i(\alpha_0,\alpha_1,\alpha_2,....):=(i,(\alpha_0^{(0)},...,\alpha_{m_0-1}^{(0)}),(\alpha_0^{(1)},...,\alpha_{m_1-1}^{(1)}), (\alpha_0^{(2)},...,\alpha_{m_2-1}^{(2)}),...)
$$
where $\prod_{k=0}^\infty\tOm=\tOm\times\tOm\times...$ is the infinite product, and $\alpha_k^{(i)}$ is the $i$-th coefficient of $\alpha_k$.\\
Finally define mappings $f_1,...,f_4:\ell_\infty(C)\to C$ by the formula
$$
f_i(x_{\alpha_0},x_{\alpha_1},x_{\alpha_2},...):=x_{i(\alpha_0,\alpha_1,\alpha_2,...)}.
$$
\begin{remark}\emph{Note that the functions $f_1,...,f_4$ have many similarities with the ones from \cite{S}. However, thanks to the use of the notion of sets $\tOm_k,{_k}\tOm$ and $\tOm$, the current definition looks more natural.}
\end{remark}
\begin{fact}\label{fak2}In the above frame, we have
\begin{equation*}
f_1(\ell_\infty(C))\cup...\cup f_4(\ell_\infty(C))=C.
\end{equation*}
\end{fact}
\begin{proof}
It is easy to see that $\bigcup_{i=1}^4i(\prod_{k=0}^\infty\tOm)=\tilde{\Omega}$. From this we easily have that $\bigcup_{i=1}^4f_i(\ell_\infty(C))=C$.
\end{proof}
\begin{fact}\label{fak3}
In the above frame, for every $i=1,...,4$, $L_{s,q}(f_i)\leq K$.
\end{fact}
\begin{proof}
Let $\alpha=(\alpha_0,\alpha_1,\alpha_2,...),\beta=(\beta_0,\beta_1,\beta_2,...)\in \prod_{k=0}^\infty\tOm$ be distinct. For every $j\in\N^*$, set
$$
\eta_j:=\min\{k\in\N^*:\alpha_j^{(k)}\neq \beta_j^{(k)}\}
$$
where we additionally define $\min\emptyset:=\infty$.\\
Observe that if $0<\eta_j<\infty$, then $x_{\alpha_j}$ and $x_{\beta_j}$ belongs to different squares $I_{\big(\alpha_j^{(0)},...,\alpha_j^{(\eta_j)}\big)}$ and $I_{\big(\beta_j^{(0)},...,\beta_j^{(\eta_j)}\big)}$, but to the same square $I_{\big(\alpha_j^{(0)},...,\alpha_j^{(\eta_j-1)}\big)}$. Therefore $\rho(x_{\alpha_j},x_{\beta_j})\geq a_{\eta_j}$. Similarly, if $\eta_j=0$, then also $\rho(x_{\alpha_j},x_{\beta_j})\geq a_{\eta_j}$ (by $\rho$ we denote the Euclidean metric on $C$). In particular
$$
\rho_{s,q}((x_{\alpha_0},x_{\alpha_1},...),(x_{\beta_0},x_{\beta_1},...))\geq \sup\{q^ja_{\eta_j}:j\in\N^*,\;\eta_j<\infty\}.
$$
Now define 
$$
k_0:=\min\{k\in\N^*:(\alpha_0^{(k)},...,\alpha_{m_k-1}^{(k)})\neq (\beta_0^{(k)},...,\beta_{m_k-1}^{(k)})\}
$$
where again $\min\emptyset:=\infty$.\\
The case when $k_0=\infty$ means that $x_{i(\alpha_0,\alpha_1,...)}=x_{i(\beta_0,\beta_1,...)}$, and hence in this case 
\begin{equation}\label{12fil3}\rho(f_i(x_{\alpha_0},x_{\alpha_1},...),f_i(x_{\beta_0},x_{\beta_1},...)))=0\leq K\rho_{s,q}((x_{\alpha_0},x_{\alpha_1},...),(x_{\beta_0},x_{\beta_1},...)).
\end{equation}
Thus assume that $k_0<\infty$ and define
$$
k_1:=\min\{k\in\N^*:\min\{\eta_0,...,\eta_{m_k-1}\}\leq k\}.
$$
Also it is easy to see that $k_1\leq k_0$.\\
If $k_1>0$ then for $j=0,...,k_1-1$,
$$ (\alpha_0^{(j)},...,\alpha_{m_j-1}^{(j)})= (\beta_0^{(j)},...,\beta_{m_j-1}^{(j)}),$$
which means that $x_{i(\alpha_0,\alpha_1,...)}$ and $x_{i(\beta_0,\beta_1,...)}$ belongs to the same square $I_\gamma$ for some $\gamma\in {_{k_1}}\tilde{\Omega}$. Hence
$$
\rho(x_{i(\alpha_0,\alpha_1,...)},x_{i(\beta_0,\beta_1,...)})\leq \sqrt{2}p_{k_1}.
$$
If $k_1=0$, then $x_{i(\alpha_0,\alpha_1,...)},x_{i(\beta_0,\beta_1,...)}\in I_{(i)}$, so we also have
$$ 
\rho(x_{i(\alpha_0,\alpha_1,...)},x_{i(\beta_0,\beta_1,...)})\leq \sqrt{2}p_{k_1}.
$$
Thus, taking $j\in\{0,...,m_{k_1}-1\}$ with $\eta_j\leq k_1$, we have by the monotonicity of $(a_k)$ and (\ref{fin3}), 
$$
\rho(f_i(x_{\alpha_0},x_{\alpha_1},...),f_i(x_{\beta_0},x_{\beta_1},...))=\rho(x_{i(\alpha_0,\alpha_1,...)},x_{i(\beta_0,\beta_1,...)})\leq \sqrt{2}p_{k_1}=
$$
$$
=\frac{\sqrt{2}}{q^{j}}\frac{p_{k_1}}{a_{k_1}}q^{j}a_{k_1}\leq \frac{\sqrt{2}}{q^{m_{k_1}}}\frac{p_{k_1}}{a_{k_1}}q^ja_{\eta_j}\leq K\sup\{q^la_{\eta_l}:l\in\N^*,\;\eta_l<\infty\}\leq $$ $$\leq K\rho_{s,q}((x_{\alpha_0},x_{\alpha_1},...),(x_{\beta_0},x_{\beta_1},...)). 
$$
Therefore $L_{s,q}(f_i)\leq K$.
 \end{proof}
We are ready to state the main theorem of this section:
\begin{theorem}\label{fin5}
$\;$
\begin{itemize}
\item[(i)] The family $\F=\{f_1,...,f_4\}$ is a GIFS$_\infty$ such that $L_{(s,q)}(\F)<K$ (in particular, $\F$ satisfies (Q)) and $C$ is its attractor.
\item[(ii)] Assume additionally that the sequence $(m_k)$ satisfies for every $k\in\N$, 
\begin{equation}\label{fin4}
(1+m_0+m_0m_1+...+m_0\cdots m_{k-1})(k-1)-m_0\cdots m_k<-k.
\end{equation}
Then there is a probability measure $\mu$ on $C$ such that for any $m\in\N$ and any generalized weak contraction $g:C^m\to C$, we have $\mu(g(C^m))=0$. In particular, $C$ is not an attractor of any GIFS on $C$.
\end{itemize}
\end{theorem}
\begin{proof}
Facts \ref{fak2} and \ref{fak3} combined with Theorem~\ref{ft1} give (i) immediately. We will prove (ii). Let $g:C^m\to C$ be a generalized weak contraction. Then for every $k\in\N$,
$$
g(C^m)=g\left(\left(\bigcup_{\alpha_0\in{_k}\tilde{\Omega}}C\cap I_{\alpha_0}\right)\times...\times\left(\bigcup_{\alpha_{m-1}\in _k\tilde{\Omega}}C\cap I_{\alpha_{m-1}}\right)\right)=\bigcup_{\alpha_0,...,\alpha_{m-1}\in_k\tilde{\Omega}}g((C\cap I_{\alpha_0})\times...\times(C\cap I_{\alpha_{m-1}})).
$$
Hence by Fact \ref{fak1}, there are at most
$$
r_k = |{_k}\tilde{\Omega}\times...\times{_k}\tilde{\Omega}|=(|\tilde{\Omega}_0|\cdots|\tilde{\Omega}_k|)^m=(4\cdot 4^{m_0}\cdots 4^{m_0\cdots m_{k-1}})^m=4^{m(1+m_0+m_0m_1+...+m_0\cdots m_{k-1})}
$$
sequences $\beta_0,...,\beta_{r_k-1}\in{_{k+1}}\tilde{\Omega}$ such that 
$$g(C^m)\subset I_{\beta_0}\cup...\cup I_{\beta_{r_k-1}}.$$
Now by the Kolmogorov theorem, there is a unique probability measure $\mu$ on $C$ such that for every $k\in\N$ and every $\alpha\in {_k}\tilde{\Omega}$, $\mu(I_\alpha\cap C)=\frac{1}{|{_k}\tilde{\Omega}|}$.
Thus, for every $k\geq m$, we have
$$
\mu(g(C^m))\leq \frac{r_k}{|{_{k+1}}\tilde{\Omega}|}\leq \frac{4^{m(1+m_0+m_0m_1+...+m_0\cdots m_{k-1})}}{4^{1+m_0+m_0m_1+...+m_0\cdots m_{k}}}=4^{(1+m_0+m_0m_1+...+m_0\cdots m_{k-1})(m-1)-m_0\cdots m_k} \leq
$$
$$
\leq 4^{(1+m_0+m_0m_1+...+m_0\cdots m_{k-1})(k-1)-m_0\cdots m_k}<4^{-k}.
$$
Since the right side tends to $0$, we have $\mu(g(C^m))=0$.
\end{proof}
\begin{remark}\emph{Note that the proof of (i) in the above theorem is essentially the same as the proof of \cite[Theorem 9]{S}. However, it is more abstract, and adjusted to terminology used in this section.}
\end{remark}
\begin{remark}\emph{Observe that in our construction we can take the constant $K$ as small as we want. Hence Theorem \ref{fin5}(ii) shows that we can find \gifs \ with as small Lipschitz constant as we want whose attractor is a Cantor-type set which can not be obtained as a fractal of any GIFS.
}\end{remark}

\begin{remark}\emph{It seems that another example of a \gifs\ attractor which is not a GIFS attractor is just the code space $\mathbf{\Omega}$ with appropriate metric (probably $d_{(p,q)}$ or $d_{(s,q)}$). However, we find the presented example more natural since it is a Cantor set on the plane. On the other hand, from its construction we see that in some sense it is an image of an appropriate part of the code space $\mathbf{\Omega}$, since at each step of the construction we somehow restrict $\Omega_k$ to $\tilde{\Omega}_k$.}
\end{remark}


\begin{thebibliography}{Mi1}


\bibitem[BKNNS]{BKNNS} T. Banakh, W. Kubi\'s, M. Nowak, N. Novosad, F. Strobin, \emph{Contractive function systems, their attractors and metrization}, Topological Methods in Nonlinear Analysis, 46, no. 2, 1029-1066.




\bibitem[Ba]{Ba} M. F. Barnsley, \emph{Fractals everywhere}. Academic Press
Professional, Boston, MA, 1993.


\bibitem[Br]{Br} F. Browder, \emph{On the convergence of successive
approximations for nonlinear functional equations}. Nederl. Akad. Wetensch.
Proc. Ser. A 71=Indag. Math. \textbf{30} (1968) 27--35. 

\bibitem[CR]{CR} S. Crovisier, M. Rams, \emph{IFS attractors and Cantor sets}%
, Topology and Appl., 153 (2006), 1849-1859.


\bibitem[Du]{Du} D.~Dymitru, \emph{Attractors of topological iterated function system}, Annals of Spiru Haret University: Mathematics-Informatics series,  {\bf 8}:2 (2012), 11--16.





\bibitem[Ha]{Ha} M. Hata, \emph{On the structure of self-similar sets.}
Japan J. Appl. Math. 2 (1985), 381--414.

\bibitem[Hut]{Hut} J. Hutchinson, \emph{Fractals and self-similarity}. Indiana
Univ. Math. J. 30 (1981), no. 5, 713-747.

\bibitem[JJ]{JJ} J. Jachymski, I. J\'o\'zwik, \emph{Nonlinear contractive
conditions: a comparison and related problems}, Banach Center Publ., 77,
Polish Acad. Sci., 77 (2007), 123--146.

\bibitem[JMS]{JMS} J. Jachymski, \L $\;$ Ma\'slanka, F. Strobin \emph{A Fixed Point Theorem for Mappings on the $l_\infty$-Sum of a Metric Space and its Application}, Filomat 31 (2017), no. 11, 3559 - 3572.

\bibitem[JaMS]{JaMS} P. Jaros, \L $\;$ Ma\'slanka, F. Strobin , \emph{Algorithms generating images of attractors of generalized iterated function systems},  Numer. Algorithms 73 (2016), no. 2, 477--499. 

\bibitem[Ka] {Ka}A.~Kameyama, \emph{Distances on topological self-similar sets and the kneading determinants}, J. Math. Kyoto Univ. {\bf 40}:4 (2000), 601--672.





\bibitem[LT]{LT} J. Lindenstrauss, L. Tzafriri, \emph{Classical Banach spaces. I. Sequence spaces}, Ergebnisse der Mathematik und ihrer Grenzgebiete, Vol. 92. Springer-Verlag, Berlin-New York, 1977. xiii+188 pp.

\bibitem[Mat]{Mat} J.~Matkowski, \emph{Integrable solutions of functional
equations}, Diss. Math. \textbf{127} (1975) 68 pp.


\bibitem[MM1]{MM1} R. Miculescu, A. Mihail,  \emph{Applications of Fixed
Point Theorems in the Theory of Generalized IFS}, Fixed Point Theory Appl.
Volume 2008, Article ID 312876, 11 pages doi:10.1155/2008/312876


\bibitem[MM2]{MM2} R. Miculescu, A. Mihail, \emph{Generalized IFSs on
Noncompact Spaces}, Fixed Point Theory Appl. Volume 2010, Article ID 584215,
11 pages doi:10.1155/2010/584215.


\bibitem[Mi1]{M} A. Mihail, \emph{Recurrent iterated function systems}. Rev.
Roumaine Math. Pures Appl., 53 (2008), 1, 43-53.

\bibitem[Mi2]{M1} A. Mihail, \emph{The shift space for recurrent iterated
function systems}. Rev. Roumaine. Math. Pures Appl. 4 (2008), 339--355.

\bibitem[Mi3]{Mi1} A. Mihail, \emph{A topological version of iterated
function system}s, An. \c Stiin\c t. Univ. Al. I. Cuza, Ia\c si, (S.N.), Matematica, Tom LVIII, (2012) 105--120.








\bibitem[Se]{Se} N. Secelean, \emph{Generalized iterated function systems on
the space $l^\infty(X)$}, J. Math. Anal. Appl. 410 (2014), no 2, 847--858.


\bibitem[SS1]{SS1} F. Strobin, J. Swaczyna, \emph{On a certain
generalisation of the iterated function system}, Bull. Aust. Math. Soc. 87
(2013), 1, pp 37-54.

\bibitem[SS2]{SS2} F. Strobin, J. Swaczyna, \emph{A code space for a
generalized IFS}, Fixed Point Theory 17 (2016), no. 2, 477--493.

\bibitem[S]{S} F. Strobin, \emph{Attractors of GIFSs that are not attractors
of IFSs}, J. Math. Anal. Appl. 422 (2015), no 1, 99-108.
\end{thebibliography}
\end{document}